\documentclass[oneside,english]{amsart}
\usepackage[latin9]{inputenc}
\usepackage{amsthm}
\usepackage{amstext}
\usepackage{amssymb}
\usepackage{esint}
\usepackage{xargs}[2008/03/08]
\usepackage[all]{xy}

\makeatletter
\numberwithin{equation}{section}
\numberwithin{figure}{section}
\theoremstyle{plain}
\newtheorem{thm}{\protect\theoremname}[section]
  \theoremstyle{remark}
  \newtheorem{acknowledgement}[thm]{\protect\acknowledgementname}
  \theoremstyle{plain}
  \newtheorem{lem}[thm]{\protect\lemmaname}
  \theoremstyle{remark}
  \newtheorem{notation}[thm]{\protect\notationname}
  \theoremstyle{definition}
  \newtheorem{defn}[thm]{\protect\definitionname}
  \theoremstyle{plain}
  \newtheorem{prop}[thm]{\protect\propositionname}
  \theoremstyle{plain}
  \newtheorem{cor}[thm]{\protect\corollaryname}
  \theoremstyle{remark}
  \newtheorem{rem}[thm]{\protect\remarkname}
  \theoremstyle{definition}
  \newtheorem{condition}[thm]{\protect\conditionname}
  \theoremstyle{definition}
  \newtheorem{example}[thm]{\protect\examplename}

\makeatother

\usepackage{babel}
  \providecommand{\acknowledgementname}{Acknowledgement}
  \providecommand{\conditionname}{Condition}
  \providecommand{\corollaryname}{Corollary}
  \providecommand{\definitionname}{Definition}
  \providecommand{\examplename}{Example}
  \providecommand{\lemmaname}{Lemma}
  \providecommand{\notationname}{Notation}
  \providecommand{\propositionname}{Proposition}
  \providecommand{\remarkname}{Remark}
\providecommand{\theoremname}{Theorem}

\begin{document}

\title{The $p$-cyclic McKay correspondence via motivic integration}

\author{Takehiko Yasuda}

\dedicatory{Dedicated to Professor Yujiro Kawamata on the occasion of his 60th
birthday}

\address{Department of Mathematics, Graduate School of Science, Osaka University,
Toyonaka, Osaka 560-0043, Japan}

\email{takehikoyasuda@math.sci.osaka-u.ac.jp}

\thanks{This work was supported by Grants-in-Aid for Scientific Research
(22740020).}
\begin{abstract}
We study the McKay correspondence for representations of the cyclic
group of order $p$ in characteristic $p.$ The main tool is the motivic
integration generalized to quotient stacks associated to representations.
Our version of the change of variables formula leads to an explicit
computation of the stringy invariant of the quotient variety. A consequence
is that a crepant resolution of the quotient variety (if any) has
topological Euler characteristic $p$ like in the tame case. Also,
we link a crepant resolution with a count of Artin-Schreier extensions
of the power series field with respect to weights determined by ramification
jumps and the representation. 
\end{abstract}
\maketitle
\global\long\def\Spec{\mathrm{Spec}\,}
\global\long\def\AS{\mathrm{AS}}
\global\long\def\mAS{\mathrm{mAS}}
\global\long\def\RP{\mathrm{RP}}
\global\long\def\fp{\mathfrak{p}}
\global\long\def\fm{\mathfrak{m}}
\global\long\def\rj{\mathrm{rj}}
\global\long\def\cO{\mathcal{O}}
\global\long\def\bRP{\mathbf{RP}}
\global\long\def\ord{\mathrm{ord}\,}
\global\long\def\ordfnc{\mathrm{ord}}
\global\long\def\GCov#1{G\text{-}\mathrm{Cov}(#1)}
\global\long\def\GCovrep#1{G\text{-}\mathrm{Cov}^{\mathrm{rep}}(#1)}
\global\long\def\bGCovrep#1{\mathrm{G}\text{-}\mathbf{Cov}^{\mathrm{rep}}(#1)}
\global\long\def\geoeq{\sim_{\mathrm{geo}}}
\global\long\def\AA{\mathbb{A}}
\newcommandx\Hom[4][usedefault, addprefix=\global, 1=, 2=]{\mathrm{Hom}_{#1}^{#2}(#3,#4)}
\global\long\def\cX{\mathcal{X}}
\global\long\def\cE{\mathcal{E}}
\global\long\def\cJ{\mathcal{J}}
\global\long\def\LL{\mathbb{L}}
\global\long\def\Var{\mathrm{Var}}
\global\long\def\ZZ{\mathbb{Z}}
\global\long\def\QQ{\mathbb{Q}}
\global\long\def\CC{\mathbb{C}}
\global\long\def\FF{\mathbb{F}}
\global\long\def\Gal{\mathrm{Gal}}
\global\long\def\bRep{\mathrm{Rep}}
\global\long\def\cY{\mathcal{Y}}
\global\long\def\cM{\mathcal{M}}
\global\long\def\Jac{\mathrm{Jac}}
\global\long\def\Gor{\mathrm{Gor}}
\global\long\def\Ker{\mathrm{Ker}}
\global\long\def\Im{\mathrm{Im}}
\global\long\def\sht{\mathrm{sht}}
\global\long\def\cZ{\mathcal{Z}}
\global\long\def\st{\mathrm{st}}
\global\long\def\PP{\mathbb{P}}
\global\long\def\NN{\mathbb{N}}
 \global\long\def\bx{\mathbf{x}}
\global\long\def\cF{\mathcal{F}}
\global\long\def\bf{\mathbf{f}}
\global\long\def\RR{\mathbb{R}}
\global\long\def\univ{\mathrm{univ}}
\global\long\def\cD{\mathcal{D}}
\global\long\def\length{\mathrm{length}\,}
\global\long\def\sm{\mathrm{sm}}
\global\long\def\cN{\mathcal{N}}
\global\long\def\GG{\mathbb{G}}
\global\long\def\cW{\mathcal{W}}
\global\long\def\top{\mathrm{top}}
\global\long\def\fs{\mathfrak{s}}

\section{Introduction}

The McKay correspondence generally means, for a finite subgroup $G$
of $SL_{d}(\CC),$ an equality between an invariant of the representation
$G\curvearrowright\CC^{d}$ and an invariant of a crepant resolution
of the quotient variety $\CC^{d}/G$ (see \cite{Reid bourbaki}).
The aim of this paper is to make a step toward the \emph{wild McKay
correspondence}, that is, the McKay correspondence for a finite subgroup
$G\subset SL_{d}(k)$ such that the characteristic of a field $k$
divides the order of $G.$ We will study the simplest possible case
where $G$ is the cyclic group of order $p$. Gonzalez-Sprinberg and
Verdier \cite{GSp and Verdier} and Schröer \cite{Schroer} also worked
on the McKay correspondence in the wild case, but on different aspects. 

Our McKay correspondence will be formulated in a similar way as Batyrev's
one \cite{Batyrev-NonArch}, which is an equality of an orbifold invariant
of the $G$-variety $\CC^{d}$ and a stringy invariant of the quotient
variety. Denef and Loeser \cite{Denef-Loeser_McKay} gave an alternative
proof, using the motivic integration and giving a more direct link
between the invariants. We follow this approach with a stacky language
by the author \cite{Yasuda Twisted_jets,Yasuda Motivic_Over_DM}. 

Let $k$ be a perfect field of characteristic $p>0$ and $G$ the
cyclic group of order $p.$ In this paper, we will study the McKay
correspondence for a finite-dimensional $G$-representation $V.$
If for $1\le i\le p,$ $V_{i}$ denotes the indecomposable $G$-representation
of dimension $i,$ then $V$ is decomposed as $V=\bigoplus_{\lambda=1}^{l}V_{d_{\lambda}},$
$1\le d_{\lambda}\le p.$ We define a numerical invariant $D_{V}$
of $V$ by 
\[
D_{V}:=\sum_{\lambda=1}^{l}\frac{(d_{\lambda}-1)d_{\lambda}}{2}.
\]
When $D_{V}\ge p,$ a \emph{stringy motivic invariant} of the quotient
variety $X:=V/G$, denoted $M_{\st}(X),$ will be defined in the same
way as in \cite{Denef-Loeser_McKay} to be some motivic integral over
the arc space of $X.$ If $X$ admits a resolution of singularities
with simple normal crossing relative canonical divisor, then the invariant
coincides with the one defined with resolution data as in \cite{Batyrev_Gor,Batyrev-NonArch}.
The following is our main result: For a positive integer $j$ with
$p\nmid j$ or for $j=0,$ we put

\[
\sht_{V}(j):=\sum_{\lambda=1}^{l}\sum_{i=1}^{d_{\lambda}-1}\left\lfloor \frac{ij}{p}\right\rfloor .
\]

\begin{thm}[Proposition \ref{prop:explict M_st} and Corollary \ref{cor:McKay-non-pair}]
\label{thm:main in intro}If $D_{V}\ge p,$ then 
\[
M_{\st}(X)=\LL^{d}+\frac{\LL^{l-1}(\LL-1)\left(\sum_{s=1}^{p-1}\LL^{s-\sht_{V}(s)}\right)}{1-\LL^{p-1-D_{V}}}.
\]

\end{thm}
When $X$ has a crepant resolution $Y\to X,$ the theorem shows that
$D_{V}=p$ and 
\[
[Y]=\LL^{d}+\LL^{l}\sum_{s=1}^{p-1}\LL^{s-\sht_{V}(s)}.
\]
In particular, $Y$ has topological Euler characteristic $p,$ which
is, in characteristic zero, conjectured by Reid \cite{Reid bourbaki}
and proved by Batyrev \cite{Batyrev-NonArch}. We will also define
the \emph{stringy motivic invariant }of the ``projectivization''
$[X/\GG_{m}]$ of $X$ and prove that it satisfies the Poincaré duality,
which was originally proved by Batyrev \cite{Batyrev_Gor} for $\QQ$-Gorenstein
projective varieties with log terminal singularities in characteristic
zero.

The proof of Theorem \ref{thm:main in intro} is based on the motivic
integration suitably generalized to the quotient stack $\cX:=[V/G]$
and the change of variables formula for the morphism $\cX\to X.$
Following \cite{Yasuda Twisted_jets,Yasuda Motivic_Over_DM}, we will
define \emph{twisted arcs} of $\cX$ and develop the motivic integration
over the space of them. A twist of a twisted arc comes from an Artin-Schreier
extension of $k((t)),$ that is, a Galois extension of the power series
field $k((t))$ with Galois group $G\cong\ZZ/(p).$ Not only there
exist infinitely many distinct twists, but also they are parameterized
by an infinite dimensional space. This contrasts strikingly the situation
in the tame case, where we have only finitely many twists. 

Let $J_{\infty}X$ be the arc space of $X$ and $\cJ_{\infty}\cX$
the space of twisted arcs of $\cX$. Then the map $\phi:\cX\to X$
induces a map $\phi_{\infty}:\cJ_{\infty}\cX\to J_{\infty}X$, which
is bijective outside measure zero subsets. The \emph{change of variables
formula }for $\phi_{\infty}$ will be formulated as
\[
\int_{A}\LL^{F}d\mu_{X}=\int_{\phi_{\infty}^{-1}(A)}\LL^{F\circ\phi_{\infty}-\ord\Jac_{\phi}-\fs_{\cX}}d\mu_{\cX}
\]
(for details, see Theorem \ref{thm:change vars}). Here for $\gamma\in\cJ_{\infty}\cX$,
if $j$ is the ramification jump of the associated Artin-Schreier
extension of $k((t)),$ then $\fs_{\cX}(\gamma):=\sht_{V}(j).$ An
interesting consequence of Theorem \ref{thm:main in intro} is the
following: Suppose that $k$ is a finite field, and that $Y\to X$
is a crepant resolution. For each finite extension $\FF_{q}/k$ with
$q$ a power of $p,$ let $N_{q,j}$ be the number of Artin-Schreier
extensions of $\FF_{q}((t))$ with ramification jump $j$. Let $E_{0}\subset Y$
be the preimage of the origin $0\in X.$ Then we have the following
equality (Corollary \ref{cor:counting AS exts}, cf. \cite{Rose})
:
\[
\sharp E_{0}(\FF_{q})=1+\frac{p-1}{p}\sum_{j>0,\, p\nmid j}\frac{N_{q,j}}{q{}^{\sht_{V}(j)}}.
\]
This result would provide new insight into the link between the singularity
theory and the Galois theory of local fields. 

The paper is organized as follows. In Section \ref{sec:G-covers},
we will construct the moduli space of $G$-covers of the formal disk
and study its structure. In Section \ref{sec:Twisted-arcs}, we proceed
with the study of twisted arcs and jets, and their moduli spaces.
Section \ref{sec:Motivic-integration} is devoted to introducing the
motivic integration over the space of twisted arcs. Section \ref{sec:The-Change-of-vars}
contains the proof of the change of variables formula, which is the
technical heart of the paper. In Section \ref{sec:Stringy-invariants},
we will define stringy invariants and conclude various versions of
the McKay correspondence from the change of variables formula. Finally
in Section \ref{sec:Comments}, we end with remarks on future problems.
\begin{acknowledgement}
The author wishes to express his thanks to Masayuki Hirokado, Yujiro
Kawamata and Shinnosuke Okawa for conversations helpful in understanding
subtle behavior of the compound $E_{6}^{1}$-singularity. 
\end{acknowledgement}

\subsection{Convention}

Throughout the paper, we work over a perfect field $k$ of characteristic
$p>0.$ A \emph{ring }means a commutative $k$-algebra. Also every
field is supposed to contain $k.$ We denote by $G$ the cyclic group
of order $p,$ $\ZZ/(p).$ We fix a generator $\sigma$ of $G$ to
be the class of $ $$1\in\ZZ.$

\section{$G$-covers of the formal disk\label{sec:G-covers}}

The main objective of this section is to construct the moduli spaces
of $G$-covers of the formal disk $D:=\Spec k[[t]]$. This will be
used in the next section in the construction of the moduli space of
twisted arcs.

\subsection{$G$-covers of the formal punctured disk\label{sub:G-covers-punctured}}

Let $D^{*}:=\Spec k((t))$ be the formal punctured disk. We will first
examine the set of étale $G$-covers of $D^{*}$, denoted by $\GCov{D^{*}}$.
It is classified by the étale cohomology group $H^{1}(D^{*},G)$ (see
\cite[page 127]{Milne book}). Then from the Artin-Schreier sequence
of étale sheaves,
\[
0\to G\to\cO_{D^{*}}\xrightarrow{\wp:f\mapsto f^{p}-f}\cO_{D^{*}}\to0,
\]
we have
\[
H^{1}(D^{*},G)=\mathrm{coker}(H^{0}(\cO_{D^{*}})\xrightarrow{\wp}H^{0}(\cO_{D^{*}})).
\]
Consequently we have the 1-to-1 correspondence
\[
\GCov{D^{*}}\leftrightarrow\frac{k((t))}{\wp(k((t)))}.
\]
More explicitly, this correspondence is described as follows: For
a ring $A$ and $f\in A,$ we define a ring extension 
\[
A[\wp^{-1}f]:=\frac{A[u]}{(u^{p}-u+f)}
\]
endowed with the $G$-action by $\sigma(u)=u+1.$ Then the $G$-cover
corresponding to the class of $f\in k((t))$ is 
\[
E_{f}^{*}:=\Spec k((t))[\wp^{-1}f].
\]

Next we will describe the set $k((t))/\wp(k((t))).$ Given $f\in k((t)),$
we denote by $f_{i}$ the coefficient of $t^{i}$ in $f$ so that
$f=\sum_{i\in\mathbb{Z}}f_{i}t^{i}$ with $f_{i}=0$ for $i\ll0.$
\begin{lem}
We have $\wp(k[[t]])=\wp(k)\cdot1\oplus k[[t]]\cdot t.$ In particular,
if $k$ is algebraically closed, then $k[[t]]=\wp(k[[t]]).$\end{lem}
\begin{proof}
For $f\in k[[t]],$ we have 
\[
\wp(f)=\sum_{p\nmid i}-f_{i}t^{i}+\sum_{p\mid i}(f_{i/p}^{p}-f_{i})t^{i}.
\]
Hence $\wp(k[[t]])\subset\wp(k)\cdot1\oplus k[[t]]\cdot t.$ For the
converse, let $g\in\wp(k)\cdot1\oplus k[[t]]\cdot t$. Then we can
inductively choose the coefficients $f_{i}$ of $f$ such that $\wp(f)=g$
as follows: First put $f_{0}=\wp^{-1}g_{0}.$ If we have chosen $f_{0},f_{1},\dots,f_{i-1}$
such that $\wp(f)\equiv g\mod t^{i},$ then we set either $f_{i}:=-g_{i}$
or $f_{i}:=f_{i/p}^{p}-g_{i}$ depending on whether $p$ divides $i$.
This shows the first assertion. The second assertion follows from
the fact that if $k$ is algebraically closed, then $\wp(k)=k.$ \end{proof}
\begin{notation}
We put $\NN':=\{j\in\ZZ\mid j>0,\, p\nmid j\}$ and $\NN'_{0}:=\NN'\cup\{0\}.$ \end{notation}
\begin{lem}
\label{lem:representative}For $f\in k((t))$, there exists $g=\sum_{i\in\NN_{0}'}g_{-i}t^{-i}\in k[t^{-1}]\subset k((t))$
such that $f-g\in\wp(k((t))).$ Moreover such $g_{i},$ $i<0,$ are
uniquely determined and the class of $g_{0}$ in $k/\wp(k)$ is also
uniquely determined. \end{lem}
\begin{proof}
From the preceding lemma, we may eliminate the terms of positive degrees
in $f$ and assume that $f_{i}=0$ for $i>0.$ Let $pi$ $(i>0)$
be the largest multiple of $p$ such that $f_{-pi}\ne0$ if any. Then
replacing $f$ with $f-\wp(f_{-pi}^{1/p}t^{-i}),$ we get that $f_{-pi}=0$
without changing $f_{i}$ for $i<-pi$. (Since $k$ is perfect, $f_{-pi}^{1/p}$
exists in $k$.) Iterating this procedure, we eventually get a polynomial
$g$ of the desired form. 

For the uniqueness, let $g'\in k[t^{-1}]$ have the same property.
From the conditions on $g$ and $g',$ we have either $h:=g-g'\in k$
or $-\ord h\in\NN'.$ However we have $h\in\wp(k((t)))$ and every
element of $\wp(k((t)))$ of negative order has order $-pn$ with
$n$ a positive integer. $ $Thus we conclude $h\in\wp(k)$. This
shows the uniqueness of the lemma.\end{proof}
\begin{defn}
Let $A$ be a ring. A \emph{representative polynomial over $A$ }is
a Laurent polynomial of the form 
\[
f=\sum_{i\in\NN'}f_{-i}t^{-i}\in A[t^{-1}],\, f_{-i}\in A.
\]
We note that there is no constant term. We denote by $\RP_{A}$ the
set of representative polynomials over $A.$
\end{defn}
Lemma \ref{lem:representative} shows the following:
\begin{prop}
\label{prop:corresp}We have a one-to-one correspondence, 
\[
\GCov{D^{*}}\leftrightarrow\RP_{k}\times\frac{k}{\wp(k)}.
\]
In particular, if $k$ is algebraically closed, then 
\[
\GCov{D^{*}}\leftrightarrow\RP_{k}.
\]
\end{prop}
\begin{defn}
Let $\bar{k}$ be the algebraic closure of $k.$ We say that $E_{1}^{*},\, E_{2}^{*}\in\GCov{D^{*}}$
are \emph{geometrically equivalent }and write $E_{1}^{*}\geoeq E_{2}^{*}$
if their complete base changes $E_{1}^{*}\hat{\times}\bar{k}$ and
$E{}_{2}^{*}\hat{\times}_{k}\bar{k}$ are isomorphic $G$-covers of
$D^{*}\hat{\times}_{k}\bar{k}=\Spec\bar{k}((t))$. 
\end{defn}
Obviously 
\[
\GCov{D^{*}}/\geoeq\leftrightarrow\RP_{k}.
\]
If $k$ is a finite field, then $k/\wp(k)$ has $p$ elements. Hence
the quotient map 
\[
\GCov{D^{*}}\to\GCov{D^{*}}/\geoeq
\]
is a $p$-to-1 surjection. 
\begin{defn}
We say that $E^{*}\in\GCov{D^{*}}$ is \emph{representative }if $E^{*}$
is isomorphic to $E_{f}^{*}$ for $f\in\RP_{k}.$ We denote the set
of representative $G$-covers of $D^{*}$ by $\GCovrep{D^{*}}.$
\end{defn}
By construction, we have:
\begin{prop}
\label{prop:G-Cov maps}The composition 
\[
\GCovrep{D^{*}}\hookrightarrow\GCov{D^{*}}\twoheadrightarrow\GCov{D^{*}}/\geoeq
\]
is bijective. Moreover the right map is $p$-to-$1$ if $k$ is a
finite field.
\end{prop}

\subsection{The stratification by the ramification jump}

The spaces $\GCov{D^{*}}$, $\GCov{D^{*}}/\geoeq$ and $\RP_{k}$
are all infinite-dimensional. We will construct stratifications of
them with finite-dimensional strata, which will help to control these
spaces.

We say that $E^{*}\in\GCov{D^{*}}$ is \emph{trivial }if $E^{*}$
is the disjoint union of $p$ copies of $D^{*},$ equivalently if
$E^{*}$ corresponds to $0$ by the correspondence in Proposition
\ref{prop:corresp}. For a non-trivial $E^{*}\in\GCov{D^{*}},$ let
$E$ be the normalization of $D:=\Spec k[[t]]$ in $\cO_{E^{*}}$
and $\fm_{E}$ the maximal ideal of $\cO_{E}.$ Then $G$ acts on
$\cO_{E}/\fm_{E}^{i}$ for all $i\in\NN.$ 
\begin{defn}
The \emph{ramification jump }of $E^{*}$ (and of $E$), denoted by
$\rj(E^{*})=\rj(E),$ is defined as follows. If $E$ is unramified
over $D,$ then we put $\rj(E)=0.$ Otherwise $\rj(E)$ is the positive
integer $j$ such that the $G$-action on $\cO_{E}/\fm_{E}^{i}$ is
trivial if $i\le j+1$, and non-trivial if $i\ge j+2$.%
\footnote{Since $\cO_{E^{*}}/k((t))$ is a cyclic extension of prime degree,
the ramification jump is unique and equal in both lower and upper
numberings. See for instance \cite[Section 2]{Thomas Lara}.%
} We thus have a function
\[
\rj:\GCov{D^{*}}\to\mathbb{Z}_{\ge0}.
\]
\end{defn}
\begin{prop}
\label{prop:order rj corresp}Let $f\in k((t))$. Suppose $j:=-\ord f\in\NN'_{0}.$
(In particular, we can take $f$ as a representative polynomial.)
By convention, if $f=0$, then we put $j=0.$ Then $\rj(E_{f}^{*})=j.$
In particular, the function $\rj$ takes values in $\NN'_{0}.$\end{prop}
\begin{proof}
Let $L:=\cO_{E_{f}^{*}}=k((t))[\wp^{-1}f]$ and $g:=\wp^{-1}f\in L$.
If $j=0,$ then $\cO_{E_{f}^{*}}$ is isomorphic to the product of
$p$ copies of $k((t))$ or to $k'((t))$ for an Artin-Schreier extension
$k'/k.$ Hence the assertion holds. Next we suppose $j>0$ and write
$j=pq-r,$ where $q$ and $r$ are integers with $1\le r\le p-1$.
If $v_{L}$ denotes the normalized valuation on $L,$ then 
\[
v_{L}(g)=-j=-pq+r.
\]
Let $l\in\left\{ 1,2,\dots,p-1\right\} $ be such that $lr=pc+1$
for some non-negative integer $c$ . Since 
\[
v_{L}(t^{lq-c}g^{l})=p(lq-c)-lj=(lpq-pc)-lpq+pc+1=1,
\]
$s:=t^{lq-c}g^{l}$ is a uniformizer of $L$. We now have 
\[
\sigma(s)=t^{lq-c}(g+1)^{l}=t^{lq-c}g^{l}+lt^{lq-c}g^{l-1}+\text{(higher degree terms)}.
\]
Therefore 
\[
\sigma(s)-s=lt^{lq-c}g^{l-1}+\text{(higher degree terms)}
\]
and 
\begin{align*}
v_{L}(\sigma(s)-s) & =p(lq-c)+(l-1)(-pq+r)\\
 & =p(lq-c)-lpq+pc+1+pq-r\\
 & =pq-r+1\\
 & =j+1.
\end{align*}
This proves the proposition.
\end{proof}
For $j\in\NN'_{0},$ we set 
\begin{gather*}
\GCov{D^{*},j}:=\{E^{*}\in\GCov{D^{*}}\mid\rj(E^{*})=j\},\\
\GCovrep{D^{*},j}:=\{E^{*}\in\GCovrep{D^{*}}\mid\rj(E^{*})=j\},\text{ and}\\
\RP_{k,j}:=\{f\in\RP_{k}\mid\ord f=-j\}.
\end{gather*}

\begin{prop}
For $j\in\NN'$, we have 
\[
\GCovrep{D^{*},j}\leftrightarrow\GCov{D^{*},j}/\geoeq\leftrightarrow\RP_{k,j}\leftrightarrow k^{*}\times k^{j-1-\left\lfloor j/p\right\rfloor }.
\]
Here $\left\lfloor \cdot\right\rfloor $ denotes the floor function,
which assigns to a real number $a$ the largest integer not exceeding
$a.$ \end{prop}
\begin{proof}
The left and middle correspondences are clear. We note that 
\[
\sharp\{i\in\NN'\mid i\le j\}=j-\left\lfloor j/p\right\rfloor .
\]
The right correspondence sends $\sum_{i\in\NN',\, i\le j}g_{-i}t^{-i},$
$g_{-j}\ne0$ to $(g_{-i})_{i\in\NN',i<j}.$ 
\end{proof}
Thus for instance, the infinite-dimensional space $\GCovrep{D^{*}}$
admits a stratification
\[
\GCovrep{D^{*}}=\bigsqcup_{j\in\NN_{0}'}\GCovrep{D^{*},j},
\]
whose strata are all finite-dimensional.

For later use, also we define $\GCov{D^{*},\le j}:=\bigcup_{j'\le j}\GCov{D^{*},j'}$
and similarly for $\GCovrep{D^{*},\le j}$ and $\RP_{k,\le j}.$ Then 

\[
\GCovrep{D^{*},\le j}\leftrightarrow\GCov{D^{*},\le j}/\geoeq\leftrightarrow\RP_{k,\le j}\leftrightarrow k^{j-\left\lfloor j/p\right\rfloor }.
\]

\subsection{Moduli spaces of $G$-covers of $D^{*}$}

Harbater \cite[Sec. 2]{Harbater} constructed the coarse moduli space
of $G$-covers of the formal disk $D=\Spec k[[t]]$ when $k$ is algebraically
closed.%
\footnote{Precisely he constructed the coarse moduli space of \emph{pointed
}principal $G$-covers. In our case where $G$ is abelian, it is equal
to the coarse moduli space of \emph{unpointed }principal $G$-covers.%
} He also illustrates with an example why the moduli space cannot have
a universal family (Ibid, Remark 2.2). Since we would like to still
have a ``universal family'' and work over a non-algebraically closed
field, we will take a different approach. 

In view of his example, to have a universal family, it seems that
we need an additional structure on $G$-covers. We will take representative
polynomials as such a structure, or rather consider the moduli space
of representative polynomials. 

For each $j$, the functor 
\[
\text{\{affine \ensuremath{k}-scheme\}\ensuremath{\to\text{\{set\}}}},\,\Spec A\mapsto\RP_{A,\le j}
\]
is obviously represented by a scheme isomorphic to $\mathbb{A}_{k}^{j-\left\lfloor j/p\right\rfloor }$,
which we denote by $\bRP_{k,\le j}.$ Explicitly we can write its
coordinate ring as 
\[
B_{\le j}:=k[x_{i}\mid i\in\NN',\, i\le j].
\]
Then the identity morphism of $\bRP_{k,\le j}$ corresponds to the
\emph{universal representative polynomial, }
\[
f_{j}^{\univ}:=\sum_{j'\in\NN',\, j'\le j}x_{j'}t^{-j'}\in\RP_{B_{\le j},\le j}.
\]
For $j_{1}\le j_{2},$ we have a canonical closed embedding $\bRP_{k,\le j_{1}}\hookrightarrow\bRP_{k,\le j_{2}}.$

Then the functor
\[
\text{\{affine \ensuremath{k}-scheme\}\ensuremath{\to\text{\{set\}}}},\,\Spec A\mapsto\RP_{A}
\]
is represented by $\bRP_{k}:=\bigcup_{j}\bRP_{k,\le j}$ in the sense
that for each affine $k$-scheme $\Spec A,$ 
\[
\RP_{A}\leftrightarrow\{\phi\in\Hom{\Spec A}{\bRP_{k}}\mid\phi(\Spec A)\subset\bRP_{k,\le j}\,(j\gg0)\}.
\]
In particular, for a field $K$, we have $\RP_{K}\leftrightarrow\bRP_{k}(K).$
The identity morphism of $\bRP_{k}$ does correspond to not a representative
\emph{polynomial }but a representative \emph{series 
\[
f_{\infty}^{\univ}:=\sum_{j\in\NN'}x_{j}t^{-j}.
\]
}
\begin{defn}
A \emph{representative family of $G$-covers of $D^{*}$ (of ramification
jump $\le j$) over an affine scheme $S=\Spec A$ }is an étale $G$-torsor
over $S\hat{\times}D^{*}$ which is isomorphic to $\Spec A((t))[\wp^{-1}f]$
with $f\in\RP_{A}$ $(f\in\RP_{A,\le j}).$ We denote the set of isomorphism
classes of those families by $\GCovrep{D^{*}}(S)$ ($\GCovrep{D^{*},\le j}(S)$).
\end{defn}
The functor
\begin{eqnarray*}
\text{\{affine \ensuremath{k}-scheme\}} & \to & \text{\{set\}}\\
S & \mapsto & \GCovrep{D^{*},\le j}(S)
\end{eqnarray*}
is represented by a scheme canonically isomorphic to $\bRP_{k,\le j}$,
which is denoted by $\bGCovrep{D^{*},\le j}$$.$ We have the universal
family of representative $G$-covers of ramification jump $\le j$:
\[
\xymatrix{E_{\le j}^{*,\univ}:=\Spec B_{\le j}((t))[\wp^{-1}f_{j}^{\univ}]\ar[d]\\
\Spec B_{\le j}((t))\ar[d]\\
\Spec B_{\le j}=\bGCovrep{D^{*},\le j}
}
\]
We define 
\[
\bGCovrep{D^{*}}:=\bigcup_{j\in\NN'_{0}}\bGCovrep{D^{*},\le j}.
\]
This represents the functor $S\mapsto\GCovrep{D^{*}}(S)$ in the same
way as $\bRP_{k}$ represents $\Spec A\mapsto\RP_{A}.$ In particular,
for a perfect field $K$, we have a one-to-one correspondence 
\[
\bGCovrep{D^{*}}(K)=\GCovrep{D^{*}\hat{\times}_{k}K}\leftrightarrow\GCov{D^{*}\hat{\times}_{k}K}/\geoeq.
\]
The universal family $E_{\infty}^{*,\univ}$ over $\bGCovrep{D^{*}}$
is defined as the union of $E_{\le j}^{*,\univ}$$.$

Putting $\bRP_{k,j}:=\bRP_{k,\le j}\setminus\bRP_{k,\le j-1},$ we
have a stratification $\bRP_{k}=\bigsqcup_{j}\bRP_{k,j}.$ Similarly
we have $\bGCovrep{D^{*}}=\bigsqcup_{j}\bGCovrep{D^{*},j}.$ Then
for $j>0,$ 
\begin{equation}
\bRP_{k,j}\cong\bGCovrep{D^{*},j}\cong\mathbb{G}_{m}\times\mathbb{A}_{k}^{j-1-\left\lfloor j/p\right\rfloor }.\label{eq:bRP dim}
\end{equation}
For a ring $A,$ the $A$-points of $\bRP_{k,j}$ corresponds to 
\[
\RP_{A,j}:=\{f\in\RP_{A,\le j}\mid f_{j}\in A^{*}\}.
\]

\subsection{The stratified moduli space of $G$-covers of the formal disk\label{sub:Moduli-disk}}

What we will really need is the moduli space of (ramified) $G$-covers
of the formal \emph{non-punctured} disk $D=\Spec k[[t]].$ A\emph{
$G$-cover }of $D$ means the normalization $E$ of $\Spec k[[t]]$
in a $G$-cover $E^{*}\to D^{*}.$ If exists, such a moduli space
should bijectively correspond with the moduli space of $G$-covers
of $D^{*}$ at the level of points. The author does not know so far
if such a moduli space exists. Instead we will construct strata of
the hypothetical moduli space, which are sufficient for application
to the motivic integration. 

We define $\bGCovrep{D,j}$ to be $\bGCovrep{D^{*},j}$ endowed with
a different universal family constructed as follows. The coordinate
ring of this moduli space is $B_{j}:=B_{\le j}[x_{j}^{-1}].$ Then
the universal family of $\bGCovrep{D^{*},j}$ is written as:
\[
E_{j}^{*,\univ}:=\Spec B_{j}((t))[\wp^{-1}f_{j}^{\univ}]\to\Spec B_{j}((t))\to\bGCovrep{D^{*},j}.
\]

Let $g:=\wp^{-1}f_{j}^{\univ}\in B_{j}((t))[\wp^{-1}f_{j}^{\univ}].$
With the notation in the proof of Proposition \ref{prop:order rj corresp},
we put $s:=t^{lq-c}g^{l}$. Then $s$ is a uniformizer on each fiber
of the projection $E_{j}^{*,\univ}\to\bGCovrep{D^{*},j}$. We define
$C_{j}$ to be the $B_{j}[[t]]$-subalgebra of $B_{j}((t))[g]$ generated
by $s.$ Then $\Spec C_{j}\to\Spec B_{j}[[t]]$ is a family of $G$-covers
of $D$\emph{ }over\emph{ $\bGCovrep{D^{*},j}.$}
\begin{defn}
We define the \emph{moduli space of representative $G$-covers of
$D$ of ramification jump $j$, }denoted by $\bGCovrep{D,j},$\emph{
}to be $\Spec B_{j}$ with the \emph{universal family} 
\[
E_{j}^{\univ}:=\Spec C_{j}\to\Spec B_{j}[[t]]\to\bGCovrep{D,j}.
\]

\end{defn}

\subsection{Details of the $G$-actions on $\cO_{E^{*}}$ and $\cO_{E}$\label{subsec: detail act}}

Let $0\ne f\in\RP_{k}$ be a representative polynomial of order $-j$.
Let $E$ and $E^{*}$ the corresponding $G$-covers of $D$ and $D^{*}$
respectively, and let $g=\wp^{-1}f\in\cO_{E^{*}}$. Then $\cO_{E^{*}}$
has a basis $1,g,\dots,g^{p-1}$ over $k((t)).$ 
\begin{notation}
In what follows, for a ring or module $M$ endowed with a $G$-action,
we denote by $\delta$ the $k$-linear operator $\sigma-\mathrm{id}_{M}$
on $M$. For $a\in\NN$, we denote by $M^{\delta^{a}=0}$ the kernel
of $\delta^{a}:M\to M.$
\end{notation}
Sometimes it is more useful to use $\delta$ rather than $\sigma$
in order to study $G$-actions. 
\begin{lem}
\label{lem:delta 1}For any integer $i$ with $1\le i\le p-1$ and
for any $0\ne h\in k((t)),$ we have $\delta^{i}(g^{i}h)\ne0$ and
$\delta^{i+1}(g^{i}h)=0.$ Therefore, for each integer $a$ with $0\le a\le p,$
we have
\[
\cO_{E^{*}}^{\delta^{a}=0}=\bigoplus_{i=0}^{a-1}k((t))\cdot g^{i}.
\]
\end{lem}
\begin{proof}
We will prove this by induction on $i.$ For $i=1,$ since $\sigma(g)=g+1,$
we have $\delta(gh)=h(\sigma(g)-g)=h$ and $\delta^{2}(gh)=\delta(h)=0.$
For $i>1,$ we have
\[
\sigma(g^{i}h)=h(g+1)^{i}=h(g^{i}+ig^{i-1}+\cdots+ig+1)
\]
and
\begin{equation}
\delta(g^{i}h)=h(ig^{i-1}+\cdots ig+1).\label{eq:delta}
\end{equation}
Applying $\delta^{i-1}$ and $\delta^{i}$ to this, we obtain the
lemma.\end{proof}
\begin{cor}
We have 
\[
\cO_{E}=\prod_{\substack{0\le i<p\\
-ij+np\ge0
}
}k\cdot g^{i}t^{n}.
\]
 Moreover for each integer $a$ with $0\le a\le p,$ we have 
\[
\cO_{E}^{\delta^{a}=0}=\prod_{\substack{0\le i<a\\
-ij+np\ge0
}
}k\cdot g^{i}t^{n}.
\]
\end{cor}
\begin{proof}
Let $v$ be the normalized valuation on $\cO_{E^{*}}.$ Then $v(g^{i}t^{n})=-ij+np.$
For every non-negative integer $r,$ there exists a unique pair $(i,n)$
of integers such that $0\le i<p$ and $r=-ij+np.$ This proves the
first assertion. Then the second follows from the preceding lemma. \end{proof}
\begin{cor}
For $h\in\cO_{E}$ with $p\nmid v_{E}(h),$ we have $v_{E}(\delta(h))=v_{E}(h)+\rj(E).$
Here $v_{E}$ denotes the normalized valuation of $\cO_{E}.$\end{cor}
\begin{proof}
We can write $h$ as a $k((t))$-linear combination of $g^{i}$, $0\le i<p.$
Then the corollary follows from equation (\ref{eq:delta}).
\end{proof}

\section{Twisted arcs and jets\label{sec:Twisted-arcs}}

To a $G$-representation $V$, we will associate the quotient stack
$\cX=[V/G]$ and the quotient variety $X=V/G.$ The McKay correspondence
follows from the change of variables formula of motivic integrals
for the morphism $\cX\to X.$ To obtain the formula, we need an almost
bijection between the arc spaces of $X$ and $\cX.$ However general
arcs of $X$ lifts to $\cX$ not as ordinary arcs but as \emph{twisted
arcs}. In this section, we will construct the spaces of twisted arcs
and jets, and examine their structures. Our use of stacks is not really
necessary. However it put everything on an equal footing in the framework
of the birational geometry of stacks.

\subsection{Ordinary arcs and jets of a scheme}

Let $X$ be a variety, that is, a separated scheme of finite type
over $k.$ An \emph{$n$-jet} of $X$ is a morphism $\Spec k[[t]]/(t^{n+1})\to X.$
There exists a fine moduli scheme $J_{n}X$ of $n$-jets of $X,$
called the \emph{$n$-jet scheme} of $X.$ Thus, for a ring $A,$
\[
(J_{n}X)(A)=\Hom{\Spec A[[t]]/(t^{n+1})}X.
\]
 There is a natural morphism $J_{n}X\to X.$ Also for $n'\ge n,$
we have a \emph{ truncation map} $J_{n'}X\to J_{n}X.$ The projective
limit, $J_{\infty}X:=\underset{n\to\infty}{\lim}J_{n}X$, is called
the \emph{arc space} of $X.$ For every field $K,$ $ $
\[
(J_{\infty}X)(K)=\Hom{\Spec K[[t]]}X.
\]
We denote the \emph{truncation map $J_{\infty}X\to J_{n}X$ by $\pi_{n}.$ }

\subsection{\label{sub:Our-setting}A $G$-representation}

From now on, we denote by $V$ a $d$-dimensional $G$-representation
and suppose that $V$ is decomposed into indecomposables as 
\[
V=\bigoplus_{\lambda=1}^{l}V_{d_{\lambda}}\,(1\le d_{\lambda}\le p,\,\sum_{\lambda=1}^{l}d_{\lambda}=d),
\]
where $V_{a}$ denotes the unique indecomposable $G$-representation
of dimension $a$. We suppose that $V$ is non-trivial, that is, $(d_{1},\dots,d_{l})\ne(1,\dots,1).$ 

We denote the coordinate ring of the affine space $V$ by 
\[
k[\bx]=k[x_{\lambda,i}\mid1\le\lambda\le l,\,1\le i\le d_{\lambda}]
\]
and fix the $G$-action on it by: 
\[
\sigma(x_{\lambda,i})=\begin{cases}
x_{\lambda,i}+x_{\lambda,i+1} & (j\ne d_{\lambda})\\
x_{\lambda,d} & (j=d_{\lambda})
\end{cases}
\]
This is equivalent to saying that:
\[
\delta(x_{\lambda,i})=\begin{cases}
x_{\lambda,i+1} & (i\ne d_{\lambda})\\
0 & (i=d_{\lambda})
\end{cases}
\]

Most arguments below can be reduced to the case where $V$ is indecomposable.
In that case, $l=1$ and $d_{1}=d.$ Then we simply write $x_{i}=x_{1,i}$.

\subsection{$G$-arcs and jets.}

For $0\ne f\in\RP_{k}$, we define $E_{f,n}$ to be $\Spec\cO_{E_{f}}/\fm_{E_{f}}^{np+1},$
which is a closed subscheme of $E_{f}.$ Since $k[[t]]/(t^{n+1})\subset(\cO_{E}/\fm_{E_{f}}^{pn+1})^{G}$
(the equality does not generally hold), we have a natural morphism
\[
E_{f,n}\to D_{n}:=\Spec k[[t]]/(t^{n+1}).
\]
If $f=0,$ then $E_{f}$ has $p$ connected components and each component
is identified with $D$ via the projection $E_{f}\to D:=\Spec k[[t]].$
In this case, we just define $E_{f,n}$ to be the disjoint union of
$p$ copies of $D_{n}.$ 
\begin{defn}
We define a \emph{$G$-arc} (resp. \emph{$G$-$n$-jet}) of $V$ as
a $G$-equivariant morphism $E_{f}\to V$ (resp. $E_{f,n}\to V$)
for some $f\in\RP_{k}.$ More generally, for $f\in\RP_{A,j}$, let
$E_{f}\to\Spec A[[t]]$ be the corresponding $G$-cover and let $E_{f,n}\subset E_{f}$
be as above. Then we define a \emph{$G$-arc of $V$ of ramification
jump $j$ over $A$ }as a $G$-equivariant morphism $E_{f}\to V.$
Two $G$-arcs over $A$ are regarded as the same if the associated
representative polynomials are the same and the morphisms are the
same. Similarly for $G$-$n$-jets. (If $f\ne f',$ then two $G$-$n$-jets
$E_{f,n}\to V$ and $E_{f',n}\to V$ must always be distinguished,
even when there is an isomorphism $E_{f,n}\cong E_{f',n}$ compatible
with morphisms to $V$ and $D_{n}.$)\end{defn}
\begin{lem}
For any ring $B$ endowed with a $G$-action, we have a bijection:
\begin{eqnarray*}
\{\text{\ensuremath{G}-equivariant ring map \ensuremath{k[\bx]\to B}}\} & \to & \prod_{\lambda=1}^{l}B^{\delta^{d_{\lambda}}=0}\\
\alpha & \mapsto & (\alpha(x_{1,1}),\dots,\alpha(x_{l,1}))
\end{eqnarray*}
\end{lem}
\begin{proof}
Let $\alpha:k[\bx]\to B$ be a $G$-equivariant ring map. Then for
every $\lambda$ and $i,$ we have $\alpha(\delta(x_{\lambda,i}))=\delta(\alpha(x_{\lambda,i})).$
In particular, $\alpha(x_{\lambda,i})=\delta^{i-1}(\alpha(x_{\lambda,1}))$
and $\delta^{d_{\lambda}}(x_{\lambda,1})=0.$ This shows that $\alpha$
is determined by $\alpha(x_{\lambda,1})$, $1\le\lambda\le l$ and
$ $the map of the lemma is well-defined.

Conversely if $(f_{1},\dots,f_{l})\in\prod_{\lambda=1}^{l}B^{\delta^{d_{\lambda}}=0}$
is given, then we define a ring map $\alpha:k[\bx]\to B$ by $\alpha(x_{\lambda,i})=\delta^{i-1}(f_{\lambda}).$
We can easily see that $\text{\ensuremath{\alpha}\ }$ is the unique
$G$-euqivariant ring map with $\alpha(x_{\lambda,1})=f_{\lambda}.$
Hence this construction gives the inverse map.\end{proof}
\begin{prop}
\label{prop:G-jet moduli}For each $0\le n<\infty$ and for each $j\in\NN'_{0}$,
there exists a fine moduli scheme $J_{n,j}^{G}V$ of $G$-$n$-jets
of $V$ of ramification jump $j.$ \end{prop}
\begin{proof}
We prove this only when $V$ is indecomposable. We first consider
the case $j>0.$ From the preceding lemma, for a fixed $f,$ $G$-$n$-jets
$E_{f}\to V$ correspond to elements of $(\cO_{E_{f}}/\fm_{E_{f}}^{np+1})^{\delta^{d}=0}.$
With the notion as in section \ref{subsec: detail act}, we have 
\begin{gather*}
\cO_{E_{f}}/\fm_{E_{f}}^{np+1}=\bigoplus_{\substack{0\le i<p\\
0\le-ij+np\le np
}
}k\cdot[g^{i}t^{n}].
\end{gather*}
Then $(\cO_{E_{f}}/\fm_{E_{f}}^{np+1})^{\delta^{d}=0}$ is the linear
subspace generated by the elements $g^{i}t^{n}$ from the basis such
that either $i<d$ or $-ij+np+dj>np.$ If we denote by $\nu_{n,j}$
the dimension of the subspace, then $G$-$n$-jets are parameterized
by $k^{\nu_{n,j}}.$ This argument can apply to families, in particular,
to the universal family over $\bGCovrep{D,j}.$ With the notation
from Section \ref{sub:Moduli-disk}, let $\fm_{j}\subset C_{j}$ be
the ideal generated by $s.$ Then $G$-$n$-jets over $\bGCovrep{D,j}$,
\[
\Spec C_{j}/\fm_{j}^{np+1}\to V,
\]
correspond to elements of $(C_{j}/\fm_{j}^{np+1})^{\delta^{d}=0},$
which is isomorphic to $B_{j}^{\nu_{n,j}}$ as a $B_{j}$-module.$ $
This shows that the desired moduli space $J_{n,j}^{G}V$ is isomorphic
to $\AA_{k}^{\nu_{n,j}}\times\bGCovrep{D,j}.$

The case where $j=0$ is easier. Then $f=0$ and $E_{0,n}$ is the
union of $p$-copies of $D_{n}.$ We fix one connected component of
$E_{0,n}$, identify it with $D_{n}$ and write $D_{n}\hookrightarrow E_{0,n}$.
Then a $G$-$n$-jet $E_{0,n}\to V$ is uniquely determined by its
restriction to $D_{n}.$ Conversely, an ordinary $n$-jet $D_{n}\to V$
uniquely extends to a $G$-$n$-jet $E_{0,n}\to V$. Therefore we
can identify $J_{n,0}^{G}V$ with $J_{n}V.$\end{proof}
\begin{prop}
The following hold:
\begin{enumerate}
\item For every $n$ and $j,$ $J_{n,j}^{G}V\cong\AA_{k}^{m}\times\bGCovrep{D,j}$
for some $m.$
\item For $n=0,$ $J_{0,j}^{G}V\cong\AA_{k}^{l}\times\bGCovrep{D,j}$ $(j\in\NN')$
and $J_{0,0}^{G}V=\AA_{k}^{d}.$ 
\item For $n'\ge n,$ truncation maps $J_{n',j}^{G}V\to J_{n,j}^{G}V$ are
induced by a (not necessarily surjective) linear map $\AA_{k}^{m'}\to\AA_{k}^{m}.$
\end{enumerate}
\end{prop}
\begin{proof}
The assertions follow from the proof of the preceding proposition.
\end{proof}
Now the space of $G$-arcs of ramification jump $j,$ denoted $J_{\infty,j}^{G}V,$
is constructed as the projective limit of $J_{n,j}^{G}V,$ $n\in\mathbb{Z}_{\ge0}.$
Hence it is isomorphic to $(\prod_{i=1}^{\infty}\AA_{k}^{1})\times\bGCovrep{D,j}$.
Let $\pi_{n}:J_{\infty,j}^{G}V\to J_{n,j}^{G}V$ denote truncation
maps. 
\begin{cor}
For $0\le n<\infty,$ 
\[
\pi_{n}(J_{\infty,j}^{G}V)\cong\begin{cases}
\AA_{k}^{nd+l}\times\bGCovrep{D,j} & (j\in\NN')\\
\AA_{k}^{(n+1)d} & (j=0).
\end{cases}
\]
Moreover the truncation map $\pi_{n+1}(J_{\infty,j}^{G}V)\to\pi_{n}(J_{\infty,j}^{G}V)$
is a trivial fibration with fiber $\AA_{k}^{d}.$\end{cor}
\begin{proof}
The case $j=0$ is obvious from $J_{n,0}^{G}V=J_{n}V.$ For $j>0,$
with the notation as in the proof of Proposition \ref{prop:G-jet moduli},
$G$-$n$-jets in $\pi_{n}(J_{\infty,j}^{G}V)$ corresponds to elements
of the linear subspace of $\cO_{E_{f}}/\fm_{E_{f}}^{np+1}$ generated
by $g^{i}t^{n}$'s with $i<d$. This shows the first assertion. The
second assertion follows from the first.\end{proof}
\begin{defn}
For $0\le n\le\infty,$ we put $J_{n}^{G}V:=\bigsqcup_{j\ge0}J_{n,j}^{G}V.$
(Here for each $j,$ $J_{n,j}^{G}V$ is a connected component of $J_{n}^{G}V.$)
\end{defn}
The following is obvious from the definition:
\begin{cor}
Truncation maps $\pi_{n+1}(J_{\infty}^{G}V)\to\pi_{n}(J_{\infty}^{G}V)$
is a trivial fibration with fiber $\AA_{k}^{d}.$ 
\end{cor}

\subsection{Twisted arcs and jets}

Let $\cX$ be the quotient stack $[V/G]$. For an algebraically closed
field $K$ and for a representative polynomial $f\in\RP_{K,j},$ we
set 
\[
\cD_{f}:=[E_{f}/G]\text{ and }\cD_{f,n}:=[E_{f,n}/G].
\]

\begin{defn}
We define a \emph{twisted arc }(resp. \emph{twisted $n$-jet}) of
$\cX$ over $K$ as a morphism 
\[
\cD_{f}\to\cX\text{ (resp. \ensuremath{\cD_{f,n}\to\cX})}
\]
which is induced from a $G$-arc $E_{f}\to V$ (resp. $G$-$n$-jet
$E_{f,n}\to V$). We say that two twisted arcs $\gamma:\cD_{f}\to\cX$
and $\gamma':\cD_{f'}\to\cX$ (over $K$) are \emph{isomorphic }if
$f=f'$ and if two morphisms $\gamma,\gamma':\cD_{f}\rightrightarrows\cX$
are 2-isomorphic. (Recall that stacks form a 2-category and hence
morphisms between two stacks form a usual category.)
\end{defn}
Clearly twisted arcs (jets) are closely related to $G$-arcs (jets).
To each $G$-arc $\gamma:E_{f}\to V$, we can associate a twisted
arc $\bar{\gamma}:\cD_{f}\to\cX$. Conversely given a twisted arc
$\cD_{f}\to\cX,$ then there exists a $G$-arc $E_{f}\to V$ whose
associated twisted arc is the given one. 
\begin{prop}
The set of twisted arcs of $\cX$ over an algebraically closed field
$K$ is in one-to-one correspondence with $(J_{\infty}^{G}V)(K)/G$
in such a way that the class of $\gamma\in(J_{\infty}^{G}V)(K)$ corresponds
to $\bar{\gamma}.$ Here $G$ acts on $J_{\infty}^{G}V$ by $\sigma(\gamma):=\gamma\circ\sigma=\sigma\circ\gamma.$
Similarly the set of twisted $n$-jets of $\cX$ over $K$ is in one-to-one
correspondence with $(J_{n}^{G}V)(K)/G.$\end{prop}
\begin{proof}
Let $\gamma_{i}:E_{f}\to V$ $(i=1,2)$ be $G$-arcs such that $ $$\bar{\gamma}:=\bar{\gamma_{1}}=\bar{\gamma_{2}}.$
We have to show that $\gamma_{1}$ and $\gamma_{2}$ are in the same
$G$-orbit. Let $E:=\cD_{f}\times_{\bar{\gamma},\cX}V.$ Then for
each $i,$ there exists an isomorphism $\alpha_{i}:E_{f}\to E$ which
fits into the following 2-commutative diagram:
\[
\xymatrix{E_{f}\ar[dr]\ar@/{}^{1pc}/[rr]^{\gamma_{i}}\ar[r]^{\alpha_{i}} & E\ar[r]\ar[d] & V\ar[d]\\
 & \cD_{f}\ar[r]^{\bar{\gamma}} & \cX
}
\]
Then $\gamma_{2}=\gamma_{1}\circ\alpha_{1}^{-1}\circ\alpha_{2}.$
For $G$-$n$-jets, the corresponding assertion holds. What remains
is to show that $\alpha_{1}^{-1}\circ\alpha_{2}=\tau$ for some $\tau\in G$.
This will be done in the following lemma in a more general setting. \end{proof}
\begin{lem}
Let $U$ be a $G$-scheme and $[U/G]$ the quotient stack with the
natural morphism $\alpha:U\to[U/G].$ Suppose that $\beta:U\to U$
is an isomorphism such that $\beta\circ\alpha$ and $\alpha$ are
isomorphic. Then $\beta=\tau$ for some $\tau\in G.$ \end{lem}
\begin{proof}
Let $m:G\times U\to U$ be the morphism defining the $G$-action and
$\pi:G\times U\to U$ the projection. From the definition of quotient
stacks, there exists a $G$-equivariant isomorphism $\epsilon:G\times U\to G\times U$
making the following diagram commutative: 
\[
\xymatrix{G\times U\ar[d]_{\pi}\ar[rr]^{m}\ar[dr]^{\epsilon} &  & U\\
U\ar[dr]_{\beta} & G\times U\ar[d]_{\pi}\ar[ur]^{m}\\
 & U
}
\]
Suppose that $\epsilon$ maps $\{1\}\times U$ onto $\{\tau\}\times U$.
If $\epsilon'$ denotes the restriction of $\epsilon$ to $\{1\}\times U$
and if we identify $\{1\}\times U$ and $\{\tau\}\times U$ with $U,$
then we have the commutative diagram:

\[
\xymatrix{U\ar[rr]^{1}\ar[dr]_{\epsilon'=\beta} &  & U\\
 & U\ar[ur]^{\tau}
}
\]
This shows that $\beta=\epsilon'=\tau^{-1}.$ \end{proof}
\begin{defn}
We define the \emph{space of twisted arcs }and\emph{ twisted $n$-jets
}of $\cX$ as the quotient schemes
\[
\cJ_{\infty}\cX:=(J_{\infty}^{G}V)/G\text{ and }\cJ_{n}\cX:=(J_{n}^{G}V)/G.
\]
Then for $0\le n\le\infty,$ we write $\cJ_{n}\cX=\bigsqcup_{j\in\NN_{0}'}\cJ_{n,j}\cX,$
where the subscript $j$ indicates ramification jumps. We define the
function 
\[
\rj:\cJ_{n}\cX\to\NN'_{0}
\]
by $\rj(\gamma):=j$ for $\gamma\in\cJ_{n,j}\cX.$\end{defn}
\begin{rem}
The genuine moduli spaces of twisted arcs or jets must be constructed
as stacks as in \cite{Yasuda Motivic_Over_DM}. \end{rem}
\begin{defn}
\label{def: univ homeo}A morphism $f:Y\to X$ of varieties is called
a \emph{universal homeomorphism }if one of the following equivalent
conditions holds:
\begin{enumerate}
\item $f$ is finite, surjective and universally injective.
\item For every morphism $X'\to X$ of schemes, the induced morphism $Y\times_{X}X'\to X'$
is a homeomorphism.
\end{enumerate}
We say that two schemes of finite type are \emph{universally homeomorphic
}if there exists a universal homeomorphism between them. For instance,
see \cite{Nicase-Sebag} for more details.
\end{defn}
If $T$ is a $G$-variety and $S\subset T$ is a $G$-stable closed
subvariety, then the map $S/G\to T/G$ is not a closed embedding but
only a universal homeomorphism onto its image. This is why this notion
is necessary below. 

We note that the $G$-action on $ $$J_{n,j}^{G}V=\AA_{k}^{m}\times\bGCovrep{D,j}$
is trivial on $\bGCovrep{D,j}$ and linear on $\AA_{k}^{m}.$ Indeed
the linearity follows from the proof of Lemma \ref{lem:delta 1}.
Hence we have the following fact which is essential to define the
motivic measure on $\cJ_{\infty}\cX$ below.
\begin{cor}
\label{cor:trivial fib}Every geometric fiber of the truncation $\pi_{n+1}(\cJ_{\infty}\cX)\to\pi_{n}(\cJ_{\infty}\cX)$
is universally homeomorphic to the quotient of $\AA_{K}^{d}$ by some
linear $G$-action with $K$ an algebraically closed field. Moreover
\[
\pi_{0}(\cJ_{\infty,j}\cX)=\cJ_{0,j}\cX=\begin{cases}
\AA^{l}\times\bGCovrep{D,j} & (j\in\NN')\\
V/G & (j=0).
\end{cases}
\]

\end{cor}

\subsection{Push-forward maps for twisted arcs and jets}

Let $X:=V/G$ be the quotient variety. Let $\phi:\cX\to X$ and $\psi:V\to X$
be the natural morphisms. For a twisted arc $\cD\to\cX,$ taking the
coarse moduli spaces, we get an arc $D\to X.$ This defines a \emph{push-forward
map} 
\[
\phi_{\infty}:\cJ_{\infty}\cX\to J_{\infty}X.
\]
 We can see that this is actually a scheme morphism as follows. Let
the solid arrows of 
\[
\xymatrix{E\ar[d]\ar[r] & V\ar@{-->}[d]^{\psi}\\
D\hat{\times}J_{\infty}^{G}V\ar[d]\ar@{-->}[r]_{\alpha} & X\\
J_{\infty}^{G}V
}
\]
be the universal family of $G$-arcs. Then there exists the dashed
arrow $\alpha$ which makes the whole diagram commutative. This morphism
$\alpha$ is a family of arcs of $X$ over $J_{\infty}^{G}V.$ From
the universality of $J_{\infty}X$, this induces a morphism $J_{\infty}^{G}V\to J_{\infty}X.$
Then we easily see that this factors through $\cJ_{\infty}\cX=(J_{\infty}^{G}V)/G,$
and obtain the desired morphism $\cJ_{\infty}\cX\to J_{\infty}X.$
\begin{notation}
From now on, for $\gamma$ in $J_{\infty}X$ or $\cJ_{\infty}\cX$,
we denote by $\gamma_{n}$ its truncation at level $n$: $\gamma_{n}=\pi_{n}(\gamma).$
\end{notation}
Let $\gamma:\cD\to\cX$ be a twisted arc and $\gamma_{n}:\cD_{n}\hookrightarrow\cD\to\cX$
its truncation at level $n.$ Then we have an arc $\phi_{\infty}\gamma:D\to X$
and its truncation at level $n$, $(\phi_{\infty}\gamma)_{n}:D_{n}\hookrightarrow D\to X.$
This $n$-jet of $X$ is depends only on $\gamma_{n}$, hence we have
a \emph{push-forward} \emph{map} 
\[
\phi_{n}:\pi_{n}(\cJ_{\infty}\cX)\to J_{n}X,\,\gamma_{n}\mapsto(\phi_{\infty}\gamma)_{n}.
\]
We easily see that this is a scheme morphism and compatible with truncation
maps. 
\begin{rem}
Unlike the tame case, we do not have a map $\cJ_{n}\cX\to J_{n}X$.
It is because $D_{n}$ is not the coarse moduli space of $\cD_{n}$.
\end{rem}
Let $V^{G}\subset V$ be the fixed point locus and $\cY:=[V^{G}/G]\subset\cX$.
Since we have supposed that $V$ is non-trivial, $\phi$ is proper
and birational. Then $\cY$ is the exceptional locus of $\phi.$ We
define $\cJ_{\infty}\cY$ to be the subset of $\cJ_{\infty}\cX$ consisting
of those twisted arcs factors through $\cY.$ Let $Y\subset X$ be
the image of $\cY.$ Then the arc space $J_{\infty}Y$ of $ $$Y$
is regarded as a subscheme of $J_{\infty}X.$
\begin{prop}
\label{prop:almost bijective}The map 
\[
\phi_{\infty}:\cJ_{\infty}\cX\setminus\cJ_{\infty}\cY\to J_{\infty}X\setminus J_{\infty}Y
\]
is bijective.\end{prop}
\begin{proof}
We will show that $\gamma\in\cJ_{\infty}\cX\setminus\cJ_{\infty}\cY$
can be reconstructed from $\bar{\gamma}:=\phi_{*}\gamma.$ Let $E^{*}\to D^{*}$
be a $G$-cover obtained as the base change of $V\to X$ by $\bar{\gamma}|_{D^{*}}.$
If $E$ is the normalization of $D$ in $E^{*},$ then the morphism
$E^{*}\to V$ uniquely extends to $E\to V,$ thanks to the valuative
criterion of properness. This is a $G$-arc and induces a twisted
arc $\cD:=[E/G]\to\cX.$ Now it is straightforward to check that this
twisted arc is isomorphic to $\gamma.$
\end{proof}

\section{Motivic integration\label{sec:Motivic-integration}}

In this section, we will introduce the motivic measure on the space
of twisted arcs and define integrals relative to this. Mostly we will
just repeat materials from the literature \cite{Batyrev_Gor,Denef-Loeser_Germs,Denef-Loeser_McKay,Sebag_motivic_formal}
with a slight modification.

\subsection{The Grothendieck ring of varieties and a variant}

Let $\Var_{k}$ denote the set of isomorphism classes of $k$-varieties.
The \emph{Grothendieck ring of varieties over $k$, }denoted $K_{0}(\Var_{k}),$\emph{
}is the abelian group generated by $[Y]\in\Var_{k}$ subject to the
following relation: If $Z$ is a closed subvariety of $Y$, then $[Y]=[Y\setminus Z]+[Z].$
It has a ring structure where the product is simply defined by $[Y][Z]:=[Y\times Z].$
We denote by $\LL$ the class $[\mathbb{A}_{k}^{1}]$ of the affine
line. 

For our purpose, we also need the following relation:
\begin{condition}
\label{cond: fib}Let $f:Y\to Z$ be a morphism of varieties. If every
geometric fiber of $f$ is universally homeomorphic to the quotient
of $\AA_{K}^{n}$ with $K$ an algebraically closed field by some
linear $G$-action, then $[Y]=\LL^{n}[Z].$ \end{condition}
\begin{defn}
We define $K'_{0}(\Var_{k})$ to be the quotient of $K_{0}(\Var_{k})$
by posing Condition \ref{cond: fib}. 
\end{defn}
Let $A$ be an abelian group and let $\chi:\Var_{k}\to A$ be a map
satisfying the following property: For every variety $Z$ and every
closed subvariety $Y\subset Z,$ $\chi(Z)=\chi(Z\setminus Y)+\chi(Y).$
(Such a map is called a \emph{generalized Euler characteristic}.)
Then there exists a unique group homomorphism 
\[
K_{0}(\Var_{k})\to A
\]
through which $\chi$ factors. Additionally, suppose that for every
morphism $f:Y\to Z$ as in Condition \ref{cond: fib}, $\chi(Y)=\chi(\AA_{k}^{n})\chi(Z)$.
Then there exists a group homomorphism $K_{0}'(\Var_{k})\to A$ which
fits into the commutative diagram:
\[
\xymatrix{\Var_{k}\ar[r]\ar[dr]\ar[ddr]_{\chi} & K_{0}(\Var_{k})\ar@/{}^{2pc}/[dd]\ar@{->>}[d]\\
 & K_{0}'(\Var_{k})\ar[d]\\
 & A
}
\]
The maps $K_{0}^{(')}(\Var_{k})\to A$ are ring maps if $A$ is a
ring and if $\chi(Y)\chi(Z)=\chi(Y\times Z)$ for any $Y$ and $Z.$

\subsection{Various realizations}

\subsubsection{Counting rational points}

In this paragraph, we suppose that $k$ is a finite field $\FF_{q}$.
Then for a finite extension $\FF_{q}/k,$ associating to a variety
$X$ the number of $\FF_{q}$-points $\sharp X(\FF_{q}),$ we obtain
a map
\[
\sharp_{q}:\Var_{k}\to\ZZ,\, X\mapsto\sharp X(\FF_{q}).
\]
 This is a generalized Euler characteristic and defines 
\[
\sharp_{q}:K_{0}(\Var_{k})\to\ZZ.
\]

Let $\bar{k}$ be a fixed algebraic closure of $k$. For a variety
$Y$ over $k,$ we denote by $Y_{\bar{k}}$ the variety over $\bar{k}$
obtained from $Y$ by extension of scalars. Then, fixing a prime $l\ne p,$
we write (compactly supported) $l$-adic étale cohomology groups as
$H^{i}(Y_{\bar{k}})=H^{i}(Y_{\bar{k}},\QQ_{l})$ and $H_{c}^{i}(Y_{\bar{k}})=H_{c}^{i}(Y_{\bar{k}},\QQ_{l})$. 
\begin{lem}
For a $G$-representation $V$ of dimension $d$, we have isomorphisms
of $\Gal(\bar{k}/k)$-representations: 
\[
H_{c}^{i}((V/G)_{\bar{k}})\cong H_{c}^{i}(V_{\bar{k}})\cong\begin{cases}
\mathbb{Q}_{l}(-d) & (i=2d)\\
0 & (\text{otherwise})
\end{cases}
\]
\end{lem}
\begin{proof}
In this proof, we omit the subscript $\bar{k}.$ Let $W:=V^{G}\subset V$
be the fixed point locus and $U:=V\setminus W$. From the exact $G$-equivariant
sequence 
\begin{equation}
\cdots\to H_{c}^{i}(U)\to H_{c}^{i}(V)\to H_{c}^{i}(W)\to H_{c}^{i+1}(U)\to\cdots,\label{long-seq}
\end{equation}
we have equivariant isomorphisms 
\[
H_{c}^{i}(U)\cong\begin{cases}
H_{c}^{i}(V) & (i=2d)\\
H_{c}^{i-1}(W) & (i=2\dim W+1)\\
0 & (\text{otherwise}).
\end{cases}
\]
 We claim $H^{i}(U/G)=H^{i}(U)^{G}.$ Indeed since $U\to U/G$ is
an étale Galois covering, we have the Hochschild-Serre spectral sequence
\cite[page 105, Theorem 2.20]{Milne book},
\[
H^{i}(G,H^{j}(U,\ZZ/l^{n}))\Rightarrow H^{i+j}(U/G,\ZZ/l^{n}).
\]
Then since $\sharp G=p\ne l,$ the group cohomology groups $H^{i}(G,H^{j}(U_{\bar{k}},\ZZ/l^{n}))$
vanish for $i\ne0$ and the spectral sequence degenerates. Hence for
each $j,$ 
\[
H^{j}(U,\ZZ/l^{n})^{G}=H^{j}(U/G,\ZZ/l^{n}).
\]
Then passing to the limits and tensoring with $\QQ_{l},$ we can show
the claim. 

Now, since the $G$-action on $H_{c}^{2\dim W}(W)$ is trivial, from
the Poincaré duality, we have $H_{c}^{i}(U/G)=H_{c}^{i}(U)$ for every
$i.$ Let $\bar{W}\subset V/G$ be the image of $W.$ Then the map
$W\to\bar{W}$ is a universal homeomorphism and hence $H_{c}^{i}(W)=H_{c}^{i}(\bar{W})$
(see for instance \cite{Nicase-Sebag}). From the five lemma, the
long exact sequence 
\[
\cdots\to H_{c}^{i}(U/G)\to H_{c}^{i}(V/G)\to H_{c}^{i}(\bar{W})\to H_{c}^{i+1}(U/G)\to\cdots
\]
is isomorphic to (\ref{long-seq}). In particular $H_{c}^{i}(V/G)\cong H_{c}^{i}(V)$
for every $i.$ The lemma follows again from the Poincaré duality.\end{proof}
\begin{lem}
\label{lem:Galois-rep}Let $f:Y\to Z$ be a morphism as in Condition
\ref{cond: fib}. Then we have an isomorphism of $\Gal(\bar{k}/k)$-representations,
\[
H_{c}^{i}(Y_{\bar{k}})\cong H_{c}^{i}(Z_{\bar{k}})\otimes\QQ(-n).
\]
\end{lem}
\begin{proof}
From the preceding lemma, we have 
\[
R^{i}f_{!}\mathbb{Q}_{l}=\begin{cases}
\mathbb{Q}_{l}(-n) & (i=2n)\\
0 & (\text{otherwise})
\end{cases},
\]
which proved the lemma.\end{proof}
\begin{prop}
The map $\sharp_{q}$ factors through $K_{0}'(\Var_{k}).$\end{prop}
\begin{proof}
This follows from the preceding lemma and the Lefschetz trace formula.
\end{proof}

\subsubsection{Poincaré polynomials}

If $k$ is a finite field, then the \emph{Poincaré polynomial} of
a smooth proper variety $X$ is 
\[
P(X;T)=\sum_{i}(-1)^{i}b_{i}(X)T^{i}\in\ZZ[T].
\]
Here $b_{i}(X)=\dim H^{i}(X_{\bar{k}}).$ Following Nicaise \cite[Appendix]{Nicaise-trace},
we generalize this to any variety over any field. Indeed there exists
a map 
\[
P:K_{0}(\Var_{k})\to\ZZ[T]
\]
and for a variety $X,$ we simply write $P([X])$ as $P(X).$ Making
the variable $T$ explicit, we also write it as $P(X;T).$ Two important
property of the generalized Poincaré polynomial are as follows: Firstly,
for a variety $X,$ the degree of $P(X)$ equals twice the dimension
of $X.$ Secondly, $P(X;1)$ equals the \emph{topological Euler characteristic
\[
e_{\top}(X):=\sum_{i}\dim H_{c}^{i}(X_{\bar{k}}).
\]
}
\begin{prop}
\label{prop: Poi factor}The map $P$ factors through $K_{0}'(\Var_{k}).$\end{prop}
\begin{proof}
Let $f$ be a morphism $Z\to Y$ as in Condition \ref{cond: fib}.
We need to show that $P(Z)=P(Y)T^{2n}.$ Let $A\subset k$ be a finitely
generated $\FF_{p}$-subalgebra such that $f$ is obtained from an
$A$-morphism $f_{A}:Z_{A}\to Y_{A}$ by extension of scalars. Let
$a:\Spec\FF_{q}\to\Spec A$ be a general closed point. Let $Y_{a}$
be the fiber of $Y\to\Spec A$ over $a$ and similarly for $Z_{a}$.
Then from \cite{Nicaise-trace}, $P(Y)=P(Y_{a})$ and $P(Z)=P(Z_{a}).$
Moreover $P(Y_{a})$ is computed from the weight filtrations on $H_{c}^{i}(Y_{a}\times_{\FF_{q}}\bar{\FF_{q}}).$
Similarly for $P(Z_{a})$. From Lemma \ref{lem:Galois-rep}, 
\[
P(Z)=P(Z_{a})=P(Y_{a})T^{2n}=P(Y)T^{2n}.
\]
This proves the proposition.
\end{proof}

\subsection{Localization and completion\label{sub:Localization-and-completion}}

We need to further extend our modified Grothendieck ring $K_{0}'(\Var_{k})$.
First consider its localization by $\LL$, $\cM':=K_{0}'(\Var_{k})[\LL^{-1}].$
Then we define its \emph{dimensional completion} $\hat{\cM}'$ as
follows. Let $F^{m}\cM'$ be the subgroup of $\cM'$ generated by
$[X]\LL^{i}$ with $\dim X+i<-m.$ Then $\{F^{m}\cM'\}_{m\in\ZZ}$
is a descending filtration of $\cM'.$ We define
\[
\hat{\cM}':=\varprojlim\cM'/F^{m}\cM'.
\]
This inherits the ring structure and the filtration from $\cM'.$
For later use, we define a \emph{norm} $\left\Vert \cdot\right\Vert $
on $\hat{\cM}'$ by
\begin{eqnarray*}
\left\Vert \cdot\right\Vert :\hat{\cM}' & \to & \RR_{\ge0}\\
a & \mapsto & \left\Vert a\right\Vert :=2^{-n},
\end{eqnarray*}
where $n:=\sup\{m\mid a\in F^{m}\hat{\cM'}\}.$ 

The map $P:K_{0}'(\Var_{k})\to\ZZ[T]$ extends to 
\[
P:\hat{\cM}'\to\ZZ((T^{-1})).
\]
In this paper, every explicitly computed element $a\in\hat{\cM'}$
is a rational function in $\LL$ with integer coefficients. Then we
define its \emph{topological Euler characteristic}, $e_{\top}(a),$
by substituting $1$ for $\LL$. Similarly, if $k$ is a finite field,
then for a finite extension $\FF_{q}/k$, we define $\sharp_{q}(a)$
be substituting $q$ for $\LL$.

\subsection{Motivic integration over a variety}

We briefly review the motivic integration over singular varieties.
The original reference for the theory in characteristic zero is \cite{Denef-Loeser_Germs}.
For the positive characteristic case, see \cite{Sebag_motivic_formal}. 

Let $X$ be a reduced $k$-variety of pure dimension $d.$ 
\begin{defn}
A subset $C\subset J_{\infty}X$ is called a \emph{cylinder }if for
some $0\le n<\infty,$ $\pi_{n}(C)\subset J_{n}X$ is a constructible
subset and $C=\pi_{n}^{-1}(\pi_{n}(C)).$ A subset $C\subset J_{\infty}X$
is called \emph{stable }if for some $0\le n<\infty,$ $\pi_{n}(C)\subset J_{n}X$
is a constructible subset and for all $n'\ge n,$ the map $\pi_{n'+1}(C)\to\pi_{n'}(C)$
is a piecewise trivial fibration with fiber $\AA_{k}^{d}.$ For a
stable subset $C\subset J_{\infty}X,$ we define its \emph{measure}
by 
\[
\mu_{X}(C):=[\pi_{n}(C)]\LL^{-nd}\in\hat{\cM}'\,(n\gg0).
\]
\end{defn}
\begin{rem}
In some literature, the value of measure and hence all computations
following it differ by a factor $\LL^{d}$. \end{rem}
\begin{defn}
For an ideal sheaf $I\subset\cO_{X},$ we define a function 
\begin{eqnarray*}
\ord I:J_{\infty}X & \to & \ZZ_{\ge0}\cup\{\infty\}\\
\gamma & \mapsto & \length k[[t]]/\gamma^{-1}I.
\end{eqnarray*}

\end{defn}
The \emph{Jacobian ideal sheaf $\Jac_{X}\subset\cO_{X}$ }of $X$
is defined to be the $d$-th Fitting ideal of the sheaf of differentials,
$\Omega_{X/k}.$ This defines a closed subscheme supported on the
singular locus. For a cylinder $C\subset J_{\infty}X$ and for each
$e\in\NN,$ the subset $C\cap(\ord\Jac_{X})^{-1}(e)$ is stable, and
we define 
\[
\mu_{X}(C):=\sum_{e=0}^{\infty}\mu_{X}(C\cap(\ord\Jac_{X})^{-1}(e)).
\]
This indeed converges in $\hat{\cM}'.$
\begin{defn}
A subset $C\subset J_{\infty}X$ is \emph{measurable }if for any $\epsilon\in\RR_{>0},$
there exists a sequence of cylinders, $C_{0}(\epsilon),\, C_{1}(\epsilon),\,\dots$,
such that 
\[
C\Delta C_{0}(\epsilon)\subset\bigcup_{i\ge1}C_{i}(\epsilon)
\]
and $\left\Vert \mu_{X}(C_{i}(\epsilon))\right\Vert <\epsilon$ for
all $i\ge1.$ Here $\Delta$ denotes the symmetric difference. If
we can take $C_{0}(\epsilon)\subset C,$ we call $C$ \emph{strongly
measurable. }
\end{defn}
We define the measure of a measurable subset $C\subset J_{\infty}X$
by 
\[
\mu_{X}(C):=\lim_{\epsilon\to0}\mu_{X}(C_{0}(\epsilon))
\]
with $C_{0}(\epsilon)$ as above. This converges and the limit is
independent of the choice of $C_{0}(\epsilon).$
\begin{defn}
\label{def: measurable integrable}Let $A\subset J_{\infty}X$ be
a subset and $F:A\to\ZZ\cup\{\infty\}$ a function on it. We say that
$F$ is \emph{measurable }if $ $every fiber $F^{-1}(n)$ is measurable.
We say that $F$ is \emph{exponentially integrable}%
\footnote{In the literature, $-F$ is called exponentially integrable when the
same condition holds.%
}\emph{ }if 
\begin{enumerate}
\item $F$ is measurable,
\item $F^{-1}(\infty)$ has measure zero, and
\item for every $\epsilon>0,$ there exist at most finitely many $n\in\ZZ$
such that $\left\Vert \mu_{X}(F^{-1}(n))\right\Vert >\epsilon.$
\end{enumerate}
For an exponentially integrable function $F:J_{\infty}X\supset A\to\ZZ\cup\{\infty\},$
the \emph{integral }of $\LL^{F}$ is defined as 
\[
\int_{A}\LL^{F}d\mu_{X}:=\sum_{n\in\ZZ}\mu_{X}(F^{-1}(n))\LL^{n}\in\hat{\cM}'.
\]

\end{defn}

\subsection{Motivic integration over the quotient stack $\cX$}

Let $V$ be a $d$-dimensional non-trivial $G$-representation and
$\cX:=[V/G].$ The following arguments contain a lot of repetitions
from the preceding subsection and from the literature. However we
have to pay attention to slight differences coming from the fact that
the space of twisted arcs, $\cJ_{\infty}\cX$, is a projective limit
of inductive limits of varieties, while $J_{\infty}X$ is only a projective
limit. 
\begin{defn}
A subset $C\subset\cJ_{n}\cX$ is called \emph{constructible }if it
is a constructible subset of $\cJ_{n,\le j}\cX$ for some $j\in\NN.$
A subset $C\subset\cJ_{\infty}\cX$ is called a \emph{cylinder }if
for some $n,$ $\pi_{n}(C)$ is constructible and $C=\pi_{n}^{-1}(\pi_{n}(C)).$ 
\end{defn}
For a cylinder $C,$ we define its measure by 
\[
\mu_{\cX}(C):=[\pi_{n}(C)]\LL^{-nd}\in\hat{\cM}'\,(n\gg0).
\]
This is well-defined from Corollary \ref{cor:trivial fib}.
\begin{defn}
\label{def: measurable subset}A subset $C\subset\cJ_{\infty}\cX$
is \emph{measurable }if for any $\epsilon\in\RR_{>0},$ there exists
a sequence of cylinders, $C_{0}(\epsilon),\, C_{1}(\epsilon),\,\dots$,
such that 
\[
C\Delta C_{0}(\epsilon)\subset\bigcup_{i\ge1}C_{i}(\epsilon)
\]
and $\left\Vert \mu_{\cX}(C_{i}(\epsilon))\right\Vert <\epsilon$
for all $i\ge1.$ If we can take $C_{0}(\epsilon)\subset C,$ then
we say that $C$ is \emph{strongly measurable.}
\end{defn}
We define the measure of a measurable subset $C\subset\cJ_{\infty}\cX$
by 
\[
\mu_{\cX}(C):=\lim_{\epsilon\to0}\mu_{\cX}(C_{0}(\epsilon)).
\]
We can show that the limit is independent of the choice of $C_{0}(\epsilon)$
in the same way as the proof of \cite[Theorem 6.18]{Batyrev_Gor}
using the following lemma.
\begin{lem}
\label{lem:cylinder compact}Let $C$ and $C_{i},$ $i\in\NN,$ be
cylinders in $\cJ_{\infty}\cX.$ If $C\subset\bigcup_{i\in\NN}C_{i}$,
then for some $m\in\NN,$ we have $C\subset\bigcup_{i=0}^{m}C_{i}.$\end{lem}
\begin{proof}
The proof follows that of \cite[Lem. 4.3.7]{Sebag_motivic_formal}.
Suppose that $C\subset\cJ_{\infty,\le j}\cX.$ Then replacing $C_{i}$
with $C_{i}\cap\cJ_{\infty,\le j}\cX,$ we may suppose also $C_{i}\subset\cJ_{\infty,\le j}\cX.$
Then since $\cJ_{\infty,\le j}\cX$ is affine and hence quasi-compact,
the lemma follows from the quasi-compactness of the constructible
topology \cite[\S 7, prop. 7.2.13]{EGA-I}.
\end{proof}
We now define \emph{measurable }and \emph{exponentially integrable
functions }defined on subsets of $\cJ_{\infty}\cX$\emph{, }and the
\emph{integral }of an exponentially integrable function in the exactly
same way as in Definition \ref{def: measurable integrable}. Following
\cite{Yasuda Twisted_jets,Yasuda Motivic_Over_DM}, we define the
order function associated an ideal sheaf on $\cX$ as follows:
\begin{defn}
For a coherent ideal sheaf $I\subset\cO_{\cX}$ and for a twisted
arc $\gamma:\cD\to\cX$, we define a function $ $$\ord I:\cJ_{\infty}\cX\to\frac{1}{p}\ZZ\cup\{\infty\}$
as follows: Let $E\xrightarrow{\alpha}\cD\to D$ be the associated
$G$-cover of $D$. Then 
\[
\ord I(\gamma):=\frac{1}{p}\cdot\mathrm{length}\left(\frac{\cO_{E}}{(\gamma\circ\alpha)^{-1}I}\right).
\]
If $\cY$ is the closed substack of $\cX$ defined by the ideal sheaf
$I,$ then we write $\ord I$ also as $\ord\cY.$
\end{defn}

\subsection{Some technical results }

Here we collect technical results on the measurability and integrability
which will be needed below. 
\begin{lem}
Let $\cY\subset\cX$ be a closed substack. Then for every $n\in\frac{1}{p}\ZZ_{\ge0}$
and for every $j\in\NN_{0}'$, $(\ord\cY)^{-1}(n)\cap\cJ_{\infty,j}\cX$
is a cylinder.\end{lem}
\begin{proof}
Let $n':=\left\lceil n\right\rceil $, where $\left\lceil \cdot\right\rceil $
is the ceiling function. Let
\[
\xymatrix{E\ar[r]^{\gamma}\ar[d]_{\xi} & V\\
J_{n',j}^{G}V
}
\]
be the universal $G$-$n'$-jet of ramification jump $j$. Then we
consider the coherent sheaf $\cF:=\xi_{*}(\gamma^{*}\cO_{\cY})$ over
$J_{n',j}^{G}V.$ From the semicontinuity, 
\[
W:=\{x\in J_{n',j}^{G}V\mid\length\cF\otimes\kappa(x)\ge pn\}
\]
is a closed subset. Let $\bar{W}$ be the image of $W$ in $\cJ_{n',j}\cX.$
Then 
\[
(\ord\cY)^{-1}(\ge n)\cap\cJ_{\infty,j}\cX=\pi_{n}^{-1}(\bar{W}),
\]
which is a cylinder. Hence $(\ord\cY)^{-1}(n)\cap\cJ_{\infty,j}\cX$
is also a cylinder.\end{proof}
\begin{lem}
\label{lem:measure zero}Let $\cY\subset\cX$ be a closed substack
of positive codimension. Then the subset $\cJ_{\infty}\cY:=(\ord\cY)^{-1}(\infty)$
of $\cJ_{\infty}\cX$ is measurable and has measure zero.\end{lem}
\begin{proof}
For any $\epsilon>0,$ we choose $n$, $j$ and $n_{i}$ $(i>j)$
so that $n\gg j\gg0$ and $n_{i_{1}}\gg n_{i_{2}}$ $(i_{1}>i_{2}).$
Then 
\begin{multline*}
(\cJ_{\infty}\cY)\Delta\left(\left(\ord\cY\right)^{-1}(\ge n)\cap\cJ_{\infty,\le j}\cX\right)\\
\subset\left(\left(\ord\cY\right)^{-1}(\ge n)\cap\cJ_{\infty,\le j}\cX\right)\cup\bigcup_{i>j}\left(\left(\ord\cY\right)^{-1}(\ge n_{i})\cap\cJ_{\infty,i}\cX\right).
\end{multline*}
This shows the lemma.\end{proof}
\begin{defn}
\label{def: semiring}We define a sub-semiring $\cN\subset\hat{\cM}'$
by
\[
\cN:=\left\{ \sum_{i\in\NN}[X_{i}]\LL^{n_{i}}\in\hat{\cM}'\mid X_{i}\in\Var_{k},\,\lim_{i\to\infty}\dim X_{i}+n_{i}=-\infty\right\} .
\]
(Notice that there is no minus sign in the above series.)

We need this semiring for a technical reason. In fact, motivic measures
and motivic integrals take values in $\cN$ (or its variant added
with $\LL^{1/r}$, defined below). \end{defn}
\begin{lem}
For $a,b\in\cN,$ we have that $\left\Vert a+b\right\Vert =\max\{\left\Vert a\right\Vert ,\left\Vert b\right\Vert \}.$\end{lem}
\begin{proof}
Let us write $a=\sum_{i\in\NN}[X_{i}]\LL^{n_{i}}$ and put $n:=\max\{\dim X_{i}+n_{i}\mid i\in\NN\}.$
Then $\left\Vert a\right\Vert =2^{n}.$ This fact proves the lemma.
(In the semiring $\cN,$ we can avoid difficulties coming from cancellation
of terms.)\end{proof}
\begin{lem}
\label{lem:total-sum-convergence}For $i,j\in\NN,$ let $a_{ij}\in\cN.$
Then the following are equivalent:
\begin{enumerate}
\item For every $i,$ $\lim_{j\to\infty}\left\Vert a_{ij}\right\Vert =0$
, and $\lim_{i\to0}\left\Vert \sum_{j\in\NN}a_{ij}\right\Vert =0.$
\item For every $j,$ $\lim_{i\to\infty}\left\Vert a_{ij}\right\Vert =0$
, and $\lim_{j\to0}\left\Vert \sum_{i\in\NN}a_{ij}\right\Vert =0.$
\item For every $\epsilon>0,$ there exist at most finitely many pairs $(i,j)\in\NN^{2}$
such that $\left\Vert a_{ij}\right\Vert >\epsilon.$
\end{enumerate}
Moreover if one of the above conditions holds, then 
\[
\sum_{i,j\in\NN^{2}}a_{ij}=\sum_{i=0}^{\infty}\sum_{j=0}^{\infty}a_{ij}=\sum_{j=0}^{\infty}\sum_{i=0}^{\infty}a_{ij}.
\]
\end{lem}
\begin{proof}
Following the definition, we can translate the first condition as
follows: For every $\epsilon>0$ and every $i,$ there exist at most
finitely many $j$ with $\left\Vert a_{ij}\right\Vert >\epsilon,$
and for every $\epsilon',$ there exist at most finitely many $i$
with 
\[
\left\Vert \sum_{j=0}^{\infty}a_{ij}\right\Vert =\max\{\left\Vert a_{ij}\right\Vert \mid j\in\NN\}>\epsilon'.
\]
(The equality above follows from the preceding lemma.) Then this is
equivalent to the third condition. Similarly we can prove the equivalence
of the second and third.\end{proof}
\begin{prop}
\label{prop:integrable_pieces}Let $A_{i},$ $i\in\NN,$ be mutually
disjoint subsets of $\cJ_{\infty}\cX$ and let $A:=\bigsqcup_{i=0}^{\infty}A_{i}$.
Let $F:A\to\ZZ\cup\{\infty\}$ be a function such that for every $i,$
$F|_{A_{i}}$ is measurable. Then $F$ is exponentially integrable
if and only if for every $i,$ $F|_{A_{i}}$ is exponentially integrable
and 
\[
\lim_{i\to0}\left\Vert \int_{A_{i}}\LL^{F}d\mu_{\cX}\right\Vert =0.
\]
Moreover if it is the case, then 
\[
\int_{A}\LL^{F}d\mu_{\cX}=\sum_{i=0}^{\infty}\int_{A_{i}}\LL^{F}d\mu_{\cX}.
\]
\end{prop}
\begin{proof}
The ``if'' part: We easily see 
\[
\lim_{i\to\infty}\left\Vert \mu_{\cX}(F^{-1}(n)\cap A_{i})\right\Vert =0.
\]
It follows that $F^{-1}(n)=\bigsqcup_{i}F^{-1}(n)\cap A_{i}$ is measurable.
From the preceding lemma, we conclude that $F$ is exponentially integrable. 

The ``only if'' part: Obviously $F|_{A_{i}}$ is exponentially integrable.
Then the assertion that 
\[
\lim_{i\to0}\left\Vert \int_{A_{i}}\LL^{F}d\mu_{\cX}\right\Vert =0
\]
follows from the preceding lemma. 

The last assertion also follows from the preceding lemma.\end{proof}
\begin{lem}
Let $F:\cJ_{\infty}\cX\supset A\to\ZZ\cup\{\infty\}.$ Then $F$ is
exponentially integrable if and only if there exist measurable subsets
$A_{i},$ $i\in\NN$ such that 
\begin{enumerate}
\item $A=\bigsqcup_{i\in\NN}A_{i}$,
\item $A_{0}$ has measure zero,
\item $F$ has a finite constant value on each $A_{i}$, $i>0,$ and 
\[
\lim_{i\to\infty}\left\Vert \mu_{\cX}(A_{i})\LL^{F(A_{i})}\right\Vert =0.
\]
 
\end{enumerate}
\end{lem}
\begin{proof}
The ``only if'' part is obvious. Suppose that there exist such measurable
subsets $A_{i}.$ Then for each $n\in\ZZ\cup\{\infty\},$ 
\[
F^{-1}(n)=\bigsqcup_{F(A_{i})=n}A_{i}.
\]
Then by assumption, $F$ is exponentially integrable on $F^{-1}(n)$
and 
\[
\lim_{n\to\infty}\left\Vert \int_{F^{-1}(n)}\LL^{F}d\mu_{\cX}\right\Vert =0.
\]
From Lemma \ref{lem:total-sum-convergence}, $F$ is exponentially
integrable. \end{proof}
\begin{lem}
Let $C\subset\cJ_{\infty}\cX$ be a strongly measurable subset and
$F:C\to\ZZ\cup\{\infty\}$ an exponentially integrable function. Let
$C_{0}(\epsilon),$ $\epsilon\in\RR_{>0}$ as in Definition \ref{def: measurable subset}.
Then 
\[
\int_{C}\LL^{F}d\mu_{\cX}=\lim_{\epsilon\to0}\int_{C_{0}(\epsilon)}\LL^{F}d\mu_{\cX}.
\]
\end{lem}
\begin{proof}
We see that
\[
\lim_{\epsilon\to0}\left\Vert \int_{C\setminus C_{0}(\epsilon)}\LL^{F}d\mu_{\cX}\right\Vert =0.
\]
Now the lemma follows from the obvious equality
\[
\int_{C}\LL^{F}d\mu_{\cX}=\int_{C\setminus C_{0}(\epsilon)}\LL^{F}d\mu_{\cX}+\int_{C_{0}(\epsilon)}\LL^{F}d\mu_{\cX}.
\]

\end{proof}

\subsection{Adding fractional powers of $\LL$}

In applications, we often consider functions on arc spaces with fractional
values. For this reason, we need to add fractional powers of $\LL$
to Grothendieck rings. For a positive integer $r$, we put 
\[
\cM'_{1/r}:=\cM'[\LL^{1/r}]=\cM'[x]/(x^{r}-\LL).
\]
Then its dimensional completion $\hat{\cM}'_{1/r}$ is defined similarly.
Now for an exponentially integrable function 
\[
F:J_{\infty}X\text{ (or \ensuremath{\cJ}}_{\infty}\cX)\supset A\to\frac{1}{r}\ZZ\cup\{\infty\},
\]
its integral $\int_{A}\LL^{F}d\mu_{X}$ (or $\int_{A}\LL^{F}d\mu_{\cX}$)
is defined as an element of $\hat{\cM}'_{1/r}$. 

If $r$ divides $r',$ then we can identify $\hat{\cM}'_{1/r}$ with
a subring of $\hat{\cM}'_{1/r'}.$ Then the values of measures and
integrals are independent of which ring we consider. Therefore, in
what follows, we will not make the value ring explicit.

\subsection{Motivic integration on $\bGCovrep D$\label{sub:Motivic-integration-on-GCov}}

For a constructible subset $C$ of $\bGCovrep{D,\le j}$, we define
its \emph{measure} simply as 
\[
\nu(C):=[C]\in\hat{\cM}'.
\]

Let $F:\bGCovrep D\to\ZZ$ be a function which is constant on each
stratum $\bGCovrep{D,j}.$ Then we write $F(\bGCovrep{D,j})$ as $F(j).$
Suppose that 
\[
\lim_{j\to\infty}F(j)+j-\left\lfloor j/p\right\rfloor =-\infty
\]
Then we define the \emph{integral }of $\LL^{F}$ by: 
\begin{eqnarray*}
\int_{\bGCovrep D}\LL^{F}d\nu & = & \sum_{j\in\NN'_{0}}\nu(\bGCovrep{D,j})\LL^{F(j)}\\
 & = & (\LL-1)\LL^{-1}\sum_{j\in\NN'_{0}}\LL^{j-\left\lfloor j/p\right\rfloor +F(j)}.
\end{eqnarray*}

\begin{prop}
Let $(\cJ_{\infty}\cX)_{0}$ be the preimage of the origin by the
projection $\cJ_{\infty}\cX\to\cX.$ Let $\pi:\cJ_{\infty}\cX\to\bGCovrep D$
be the projection. Then we have 
\[
\int_{\bGCovrep D}\LL^{F}d\nu=\int_{(\cJ_{\infty}\cX)_{0}}\LL^{F\circ\pi}d\mu_{\cX}.
\]
\end{prop}
\begin{proof}
From Corollary \ref{cor:trivial fib}, the preimage of $\cJ_{0,j}\cX\to\cX$
of the origin is isomorphic to $\bGCovrep{D,j}.$ The proposition
follows from this.
\end{proof}

\section{The Change of variables formula\label{sec:The-Change-of-vars}}

We keep the notation from the preceding section: $V$ is a non-trivial
$G$-representation, $\cX=[V/G]$ and $X=V/G$. In this section, we
will prove the change of variables formula for the map $\phi_{\infty}:\cJ_{\infty}\cX\to J_{\infty}X,$
which enables us to express integrals on $\cJ_{\infty}\cX$ as ones
on $J_{\infty}X,$ and vice versa.

\subsection{Preliminary results}
\begin{defn}
Let $S=k[\bx]$ be the coordinate ring of $V$ and $R:=S^{G}$ the
one of $X.$ The \emph{Jacobian ideal} $\Jac_{\psi}\subset S$ of
the quotient map $\psi:V\to X$ is defined as the $0$-th Fitting
ideal of the module of differentials, $\Omega_{V/X}.$ This defines
a coherent ideal sheaf on $\cX$, which we call the \emph{Jacobian
ideal} of $\phi:\cX\to X$ and denote it by $\Jac_{\phi}\subset\cO_{\cX}.$ \end{defn}
\begin{lem}
Let $f=\sum_{i\ge r}a_{i}t^{i}\in k[[t]]$, $a_{r}\ne0,$ be a power
series of order $r$ and let $f^{-1}=\sum_{i\ge-r}b_{i}t^{i}\in k((t))$
be its inverse. Then the negative part, $\sum_{i=-r}^{-1}b_{i}t^{i},$
of $f^{-1}$ depends only on the class of $f$ in $k[[t]]/(t^{2r}).$\end{lem}
\begin{proof}
The classes of $t^{-r}f$ and $t^{r}f^{-1}$ in $k[[t]]/(t^{r})$
are inverses to each other. The negative part of $f^{-1}$ depends
only on the class of $t^{r}f^{-1}$ in $k[[t]]/(t^{r})$. Then it
depends only on the class of $t^{-r}f$ in $k[[t]]/(t^{r})$ and depends
only on the class of $f$ in $k[[t]]/(t^{2r}).$\end{proof}
\begin{prop}
There exists a constant $c\ge2$ depending only on $V$ such that:
If $\gamma,\gamma'\in\cJ_{\infty}\cX$ and if $\phi_{n}\gamma_{n}=\phi_{n}\gamma_{n}'$
for some $n$ with 
\[
n\ge c\cdot\ord\Jac_{\phi}(\gamma),
\]
then $\gamma$ and $\gamma'$ have the same associated $G$-covers
of $D.$ \end{prop}
\begin{proof}
From Lemmas \ref{lem:Jac=00003DJac} and \ref{lem:reflection Jac}
below, $\Jac_{\psi}$ is generated by elements of $R,$ say $f_{1},\dots,f_{l}\in R$.
Set $\bar{\gamma}:=\phi_{\infty}\gamma$ and $\tilde{f_{i}}:=\bar{\gamma}^{*}f_{i}\in k[[t]].$
Then $e:=\ord\Jac_{\phi}(\gamma)$ is equal to the minimum of $\ord\tilde{f_{i}}$,
say $\ord\tilde{f_{1}}.$ Let $S_{f_{1}}$ and $R_{f_{1}}$ be localizations
of $S$ and $R$ by $f_{1}$. Then $\Spec S_{f_{1}}\to\Spec R_{f_{1}}$
is an étale $G$-cover and we have 
\[
S_{f_{1}}=R_{f_{1}}[\wp^{-1}(g/f_{1}^{m})]
\]
for some $g\in R$ and $m\ge0.$ Then the $G$-cover $E^{*}\to D^{*}$
associated to $\gamma$ is given by 
\[
\cO_{E^{*}}=k((t))[\wp^{-1}(\tilde{g}/\tilde{f_{1}}^{m})],
\]
with $\tilde{g}:=\bar{\gamma}^{*}g$. Hence $E^{*}$ is determined
by the negative part of the Laurent power series $\tilde{g}/\tilde{f_{1}}^{m}$.
From the preceding lemma, the negative part of $(\tilde{f_{1}})^{-m}$
is determined by $\tilde{f}_{1}^{m}$ modulo $t^{2me}.$ The terms
of $\tilde{g}$ of degree $\ge me$ do not contribute to the negative
part of $\tilde{g}/\tilde{f_{1}}^{m}$. This shows that $E^{*}$ depends
only on $\phi_{n}\gamma_{n}$ if $n\ge2me.$ The proposition follows.\end{proof}
\begin{prop}
Let $c$ be a constant as the preceding proposition. Then there exists
a constant $c'>0$ depending only on $X$ such that: For $\gamma,\gamma'\in\cJ_{\infty}\cX,$
if we put $e:=\ord\Jac_{\phi}(\gamma)$ and if $\phi_{\infty}\gamma_{n}=\phi_{\infty}\gamma'_{n}$
for some $n$ with 
\[
n\ge\max\{ce,\, c'\cdot\ord\Jac_{X}(\phi_{\infty}\gamma)\},
\]
 then $\gamma_{n-e}=\gamma'_{n-e}.$\end{prop}
\begin{proof}
Let $\tilde{\gamma},\tilde{\gamma'}\in J_{\infty}^{G}V$ be liftings
of $\gamma,\gamma'$ respectively. From the preceding proposition,
$\gamma$ and $\gamma'$ have the same $G$-cover $E$ of $D$. Fixing
an isomorphism $E\cong\Spec k[[s]],$ we can think of $\tilde{\gamma},\tilde{\gamma'}$
as elements of the ordinary arc space $J_{\infty}V$ of $V.$ Let
$c'$ be the constant $c_{X}$ in \cite[Lemme 7.1.1]{Sebag_motivic_formal}.
Then there exists $\beta\in J_{\infty}V$ such that $\phi_{\infty}\beta=\phi_{\infty}\tilde{\gamma'}$
and $\beta_{pn-pe}=\tilde{\gamma}_{pn-pe}.$ (Here $\tilde{\gamma}$
is now considered as an element of $J_{\infty}V$ and so $\tilde{\gamma}_{pn}$
is the image of $\tilde{\gamma}$ by $J_{\infty}V\to J_{pn}V$, which
is the same as the image of $\tilde{\gamma}$ by $J_{\infty}^{G}V\to J_{n}^{G}V$.)
The equality $\phi_{\infty}\beta=\phi_{\infty}\tilde{\gamma'}$ shows
that $\beta$ is actually a $G$-arc and in the same $G$-orbit as
$\tilde{\gamma'}$. Then the equality $\beta_{pn-pe}=\tilde{\gamma}_{pn-pe}$
implies $\gamma_{n-e}=\gamma'_{n-e}.$
\end{proof}
We can rephrase the proposition as follows:
\begin{cor}
\label{cor: fiber contained}Let $\gamma\in\cJ_{\infty}\cX$, $e:=\ord\Jac_{\phi}(\gamma)$
and $e':=\ord\Jac_{X}(\phi_{\infty}\gamma).$ Let $c$ and $c'$ be
positive constants as above. Then if $n\ge\max\{ce,c'e'\},$ then
$\phi_{n}^{-1}(\phi_{n}\gamma_{n})$ is included in the fiber of $\pi_{n}(\cJ_{\infty}\cX)\to\pi_{n-e}(\cJ_{\infty}\cX)$
over $\gamma_{n-e}.$ 
\begin{cor}
\label{cor:image of cylinder}Let $C\subset\cJ_{\infty}\cX$ be a
cylinder with $C\cap(\ord\Jac_{\phi})^{-1}(\infty)=\emptyset.$ Then
$\phi_{\infty}(C)\subset J_{\infty}X$ is a stable subset. 
\end{cor}
\end{cor}
\begin{proof}
From Lemma \ref{lem:cylinder compact}, without loss of generality,
we may suppose that the functions $\ord\Jac_{\phi}$ and $\ord\phi^{-1}\Jac_{X}$
take constant values, say $e$ and $e'$, on $C.$ Let $n\in\NN$
be such that $C$ is a cylinder at level $n$ and $n\ge\max\{ce,c'e'\}.$
Then from the preceding corollary, $\phi_{\infty}(C)$ is a cylinder
at level $n+e.$ Moreover, since the function $\ord\Jac_{X}$ is constant
on it, $\phi_{\infty}(C)$ is stable. \end{proof}
\begin{cor}
\label{cor:image of str measurable}If $C\subset\cJ_{\infty}\cX$
is a (strongly) measurable subset, then so is $\phi_{\infty}(C)$.\end{cor}
\begin{proof}
If $\phi_{\infty}(C)\setminus J_{\infty}Y$ is strongly measurable,
so is $\phi_{\infty}(C).$ Hence we may suppose that $C$ is disjoint
from $\cJ_{\infty}\cY.$ Then, let $C_{i}(\epsilon)\subset\cJ_{\infty}\cX$
be cylinders as in Definition \ref{def: measurable subset}. Replacing
$C_{i}(\epsilon),$ $i>0$ with their intersections with $(\ord\cY)^{-1}(n)\cap\cJ_{\infty,j}\cX$,
$n\in\NN,$ $j\in\NN'_{0},$ we may suppose that $C_{i}(\epsilon)$
are all disjoint from $\cJ_{\infty}\cY.$ Then from the preceding
corollary, $\phi_{\infty}(C_{i}(\epsilon))$ are cylinders as well.
Moreover we easily see that 
\[
\left\Vert \mu_{\cX}(C_{i}(\epsilon))\right\Vert \ge\left\Vert \mu_{X}(\phi_{\infty}(C_{i}(\epsilon)))\right\Vert .
\]
This shows that $\phi_{\infty}(C)$ is strongly measurable. 
\end{proof}

\subsection{The key dimension count}

The essential part in the proof of the change of variables formula
is counting the dimension of $\phi_{n}^{-1}(\phi_{n}\gamma_{n})$
for $n\gg0.$ To do this, we will follow Looijenga's argument \cite{Looijenga}.

\subsubsection{Identifying $\phi_{n}^{-1}(\phi_{n}\gamma_{n})$ with a certain Hom
module.}

For simplicity, we first suppose that $V$ is indecomposable and that
the $G$-action on the coordinate ring $k[\bx]=k[x_{1},\dots,x_{d}]$
is given by $\sigma(x_{i})=x_{i}+x_{i+1}$ $(i<d)$ and $\sigma(x_{d})=x_{d}$.
Let $\gamma,\gamma'\in\cJ_{\infty}\cX$ be such that $\phi_{n}\gamma_{n}=\phi_{n}\gamma'_{n}$
for $n\ge\max\{ce,c'e'\}$ with the notation as above. Then we can
choose their liftings $\beta,\beta'\in J_{\infty}^{G}V$ such that
$\beta_{n-e}=\beta'_{n-e}.$ Let $E$ be the $G$-cover of $D$ associated
with $\beta$ and $\beta'.$ Then $\beta^{*}$ and $(\beta')^{*}$
induce the same $S$-module structure on $\fm_{E}^{(n-e)p+1}/\fm_{E}^{2(n-e)p+2}.$
Since $np+1\le2(n-e)p+2$, it induces an $S$-module structure on
$M_{n,e}:=\fm_{E}^{(n-e)p+1}/\fm_{E}^{np+1}.$ Then the induced map
\[
\beta^{*}-(\beta')^{*}:S\to M_{n,e}
\]
is a $k$-derivation and corresponds to an $S$-linear map 
\[
\Omega_{S/k}\to M_{n,e}
\]
and an $\cO_{E}$-linear map
\[
\Delta_{\beta,\beta'}:\beta^{*}\Omega_{S/k}\to M_{n,e},
\]
which are $G$-equivariant. Moreover, from the construction, 
\[
\Delta_{\beta,\beta'}(dx_{1})\in M_{n,e}^{\natural}:=\Im((\fm_{E}^{(n-e)p+1})^{\delta^{d}=0}\to M_{n,e}).
\]
Let $F_{n,e,\beta}$ be the fiber of the map $\pi_{n}(J_{\infty}^{G}V)\to\pi_{n-e}(J_{\infty}^{G}V)$
over $\beta_{n-e}$. Let 
\begin{multline*}
\Hom[\cO_{E}][\natural]{\beta^{*}\Omega_{S/k}}{M_{n,e}}\\
:=\left\{ \alpha\in\Hom[\cO_{E}]{\beta^{*}\Omega_{S/k}}{M_{n,e}}\mid\text{\ensuremath{\alpha}\ is \ensuremath{G}-equiv. and }\alpha(dx_{1})\in M_{n,e}^{\natural}\right\} .
\end{multline*}
Then we have an injection
\begin{eqnarray*}
F_{n,e,\beta} & \to & \Hom[\cO_{E}][\natural]{\beta^{*}\Omega_{S/k}}{M_{n,e}}\\
\beta_{n}' & \mapsto & \Delta_{\beta,\beta'}.
\end{eqnarray*}
Since a $G$-equivariant map $\alpha$ is determined by $\alpha(dx_{1}),$
$ $$\Hom[\cO_{E}][\natural]{\beta^{*}\Omega_{S/k}}{M_{n,e}}$ is
identified with $M_{n,e}^{\natural}.$ Comparing the dimensions, we
conclude that $F_{n,e,\beta}$ is identified with $M_{n,e}^{\natural}$
and with $\Hom[\cO_{E}][\natural]{\beta^{*}\Omega_{S/k}}{M_{n,e}}$. 

Consider the exact sequence
\[
0\to\Hom[\cO_{E}][\natural]{\beta^{*}\Omega_{S/R}}{M_{n,e}}\to\Hom[\cO_{E}][\natural]{\beta^{*}\Omega_{S/k}}{M_{n,e}}\to\Hom[\cO_{E}]{(\psi\circ\beta)^{*}\Omega_{R/k}}{M_{n,e}},
\]
where $\Hom[\cO_{E}][\natural]{\beta^{*}\Omega_{S/R}}{M_{n,e}}$ is
the preimage of $\Hom[\cO_{E}][\natural]{\beta^{*}\Omega_{S/k}}{M_{n,e}}$
in $\Hom[\cO_{E}]{\beta^{*}\Omega_{S/R}}{M_{n,e}}.$ Then $\beta$
and $\beta'$ have the same image in $J_{n}X$ if and only if $\Delta_{\beta,\beta'}$
maps to $0\in\Hom[\cO_{E}]{(\psi\circ\beta)^{*}\Omega_{R/k}}{M_{n,e}}.$
This shows the following:
\begin{prop}
\label{prop:identify the fiber with Hom}The fiber of the map $F_{n,e,\beta}\to J_{n}X$
over $\phi_{n}\gamma_{n}$ is identified with $\Hom[\cO_{E}][\natural]{\beta^{*}\Omega_{S/R}}{M_{n,e}}.$
Hence $\phi_{n}^{-1}(\phi_{n}\gamma_{n})$ is universally homeomorphic
to the quotient of $\Hom[\cO_{E}][\natural]{\beta^{*}\Omega_{S/R}}{M_{n,e}}$
by some $G$-linear action. In particular,
\[
[\phi_{*}^{-1}(\phi_{*}\gamma_{n})]=[\Hom[\cO_{E}][\natural]{\beta^{*}\Omega_{S/R}}{M_{n,e}}]\in\hat{\cM}'.
\]

\end{prop}
When $V$ is decomposable, we define $\Hom[\cO_{E}][\natural]{\beta^{*}\Omega_{S/k}}{M_{n,e}}$
to be the submodule of $\Hom[\cO_{E}]{\beta^{*}\Omega_{S/k}}{M_{n,e}}$
consisting of those $G$-equivariant maps $\alpha$ with $\alpha(dx_{\lambda,1})\in M_{n,e}^{\natural},$
$1\le\lambda\le l.$ Then we similarly define $\Hom[\cO_{E}][\natural]{\beta^{*}\Omega_{S/R}}{M_{n,e}}$.
Now Proposition \ref{prop:identify the fiber with Hom} holds also
in the decomposable case by the same reasoning.

\subsubsection{Counting the dimension of $\protect\Hom[\cO_{E}][\natural]{\beta^{*}\Omega_{S/R}}{M_{n,e}}$:
The indecomposable case }

We now suppose that $V$ is indecomposable. The decomposable case
will be discussed in the next subsection. To count the dimension of
$ $$\Hom[\cO_{E}][\natural]{\beta^{*}\Omega_{S/R}}{M_{n,e}}$, we
have to know a precise structure of the module $\beta^{*}\Omega_{S/R}.$
It is the quotient of a free module $\beta^{*}\Omega_{S/k}=\bigoplus_{i}\cO_{E}\cdot dx_{i}$
by the submodule $\Im((\psi\circ\beta)^{*}\Omega_{R/k}\to\beta^{*}\Omega_{S/k}).$
Then we first note that the submodule is generated by $G$-invariant
elements. Let $\omega=\sum_{i=1}^{d}\omega_{i}dx_{i}\in\beta^{*}\Omega_{S/k}$
be $G$-invariant. Then
\begin{eqnarray*}
\sigma(\omega) & = & \sum_{i=1}^{d-1}\sigma(\omega_{i})(dx_{i}+dx_{i+1})+\sigma(\omega_{d})dx_{d}\\
 & = & \sigma(\omega_{1})dx_{1}+\sum_{i=2}^{d}(\sigma(\omega_{i-1})+\sigma(\omega_{i}))dx_{i}\\
 & = & \omega.
\end{eqnarray*}
Hence $\sigma(\omega_{1})=\omega_{1}$ and $\sigma(\omega_{i-1})+\sigma(\omega_{i})=\omega_{i}$,
$i\ge2.$ 
\begin{notation}
For an abelian group $M$ endowed with a $G$-action, we define an
operator $\delta_{-}$ on $M$ by 
\[
\delta_{-}:=\sigma^{-1}-\mathrm{id}_{M}=-\sigma^{-1}\delta.
\]

\end{notation}
Then $ $$\omega_{i-1}=\delta_{-}(\omega_{i})$ and $\omega_{1}$
is $G$-invariant. Furthermore this is equivalent to 
\[
\omega_{i}=\delta_{-}^{d-i}(\omega_{d})\text{ and }\delta_{-}^{d}(\omega_{d})=0.
\]

\begin{notation}
For $f\in\cO_{E}^{\delta^{d}=0},$ we define a $G$-invariant element
$\omega_{f}\in\beta^{*}\Omega_{S/k}$ by 
\[
\omega_{f}:=\sum_{i=1}^{d}\delta_{-}^{d-i}(f)\cdot dx_{i}.
\]

\end{notation}
We note that if $h\in k[[t]],$ then we have 
\[
\omega_{hf}=h\cdot\omega_{f}.
\]

As before, we write $\cO_{E^{*}}=k((t))[\wp^{-1}f]$ and put $g:=\wp^{-1}f.$
Then $1,g,\dots,g^{p-1}$ form a basis of $\cO_{E^{*}}$ as a $k((t))$-vector
space. Hence every $f\in\cO_{E^{*}}$ is uniquely written as $f=\sum_{\lambda=0}^{p-1}f^{(\lambda)},$
$f^{(\lambda)}\in k((t))\cdot g^{\lambda}.$ Suppose that $\delta_{-}^{d}(f)=0,$
or equivalently that $f^{(\lambda)}=0$ for $\lambda\ge d.$ Then
from Lemma \ref{lem:delta 1}, we have 
\[
\omega_{f}=\sum_{i=1}^{d}\left(\sum_{\lambda\ge d-i}\delta_{-}^{d-i}(f^{(\lambda)})\right)dx_{i}.
\]

\begin{lem}
\label{lem:nice basis omega's} There exists a $\cO_{E}$-basis $\omega_{f_{1}},\dots,\omega_{f_{d}}$
of $\Im((\psi\circ\beta)^{*}\Omega_{R/k}\to\beta^{*}\Omega_{S/k})$
such that $f_{i}^{(d-i)}\ne0$ and $f_{i}^{(\lambda)}=0$, $\lambda>d-i.$
Namely for every $1\le i\le d,$ the terms of $dx_{i'},$ $i'<i,$
in $\omega_{f_{i}}$ vanish, hence we have 
\[
\begin{pmatrix}\omega_{f_{1}}\\
\omega_{f_{2}}\\
\vdots\\
\omega_{f_{d-1}}\\
\omega_{f_{d}}
\end{pmatrix}=\begin{pmatrix}\delta_{-}^{d-1}(f_{1}) & \delta_{-}^{d-2}(f_{1}) & \dots & \delta_{-}(f_{1}) & f_{1}\\
 & \delta_{-}^{d-2}(f_{2}) & \dots & \delta_{-}(f_{2}) & f_{2}\\
 &  & \ddots & \vdots & \vdots\\
 &  &  & \delta_{-}(f_{d-1}) & f_{d-1}\\
 &  &  &  & f_{d}
\end{pmatrix}\begin{pmatrix}dx_{1}\\
dx_{2}\\
\vdots\\
dx_{d-1}\\
dx_{d}
\end{pmatrix}.
\]
\end{lem}
\begin{proof}
The proof is more or less a standard linear algebra. Since $\cO_{E}$
is a PID and $\beta^{*}\Omega_{S/R}$ is a torsion-module, there exists
a basis of $\Im((\psi\circ\beta)^{*}\Omega_{R/k}\to\beta^{*}\Omega_{S/k})$
consisting of $d$ elements, say $\omega_{f_{1}},\dots,\omega_{f_{d}}.$
For some $i,$ $f_{i}^{(d-1)}\ne0$, say $f_{1}^{(d-1)}\ne0.$ Moreover
we may and shall suppose that $f_{1}^{(d-1)}$ has the least order
among nonzero $f_{i}^{(d-1)}$. Then replacing $f_{i}$, $i\ge2$
with $f_{i}-g_{i}f_{1}$ for suitable $g_{i}\in k[[t]],$ we may suppose
that $f_{i}^{(d-1)}=0$ for $i\ge2.$ Repeating this procedure to
$\omega_{f_{i}},$ $i\ge2,$ we may suppose that $f_{2}^{(d-2)}\ne0$
and $f_{i}^{(d-2)}=0$, $i\ge3.$ Repeating this, we eventually get
a basis of the expected form. \end{proof}
\begin{lem}
Let $f_{1},\dots,f_{d}$ be as above. Then 
\[
\ord\Jac_{\phi}(\gamma)=\frac{1}{p}\sum_{i=1}^{d}v_{E}(\delta_{-}^{d-i}(f_{i})).
\]
\end{lem}
\begin{proof}
Since Fitting ideals are compatible with pull-backs, the ideal, $\gamma^{-1}\Jac_{\phi}=\beta^{-1}\Jac_{\psi}\subset\cO_{E},$
is equal to the $0$-th Fitting ideal of $\beta^{*}\Omega_{S/R}$.
This module is isomorphic to the cokernel of a map $\cO_{E}^{d}\to\cO_{E}^{d}$
defined by the matrix 
\[
\begin{pmatrix}\delta_{-}^{d-1}(f_{1}) & \delta_{-}^{d-2}(f_{1}) & \dots & \delta_{-}(f_{1}) & f_{1}\\
 & \delta_{-}^{d-2}(f_{2}) & \dots & \delta_{-}(f_{2}) & f_{2}\\
 &  & \ddots & \vdots & \vdots\\
 &  &  & \delta_{-}(f_{d-1}) & f_{d-1}\\
 &  &  &  & f_{d}
\end{pmatrix}.
\]
Now, by the definition of Fitting ideals, $\beta^{-1}\Jac_{\psi}$
is the determinant of the matrix and equal to $\prod_{i=1}^{d}\delta^{d-i}(f_{i}).$
Hence 
\[
\length\cO_{E}/\gamma^{-1}\Jac_{\phi}=\sum_{i=1}^{d}v_{E}(\delta_{-}^{d-i}(f_{i})),
\]
which proves the lemma.\end{proof}
\begin{defn}
Suppose that $V$ is an indecomposable $G$-representation of dimension
$d$. Then we define the \emph{shift number of $j\in\NN'_{0}$ with
respect to $V$} to be 
\[
\sht_{V}(j):=\sum_{i=1}^{d-1}\left\lfloor \frac{ij}{p}\right\rfloor .
\]
\end{defn}
\begin{prop}
\label{prop:key dim count}Let the assumption be as in Corollary \ref{cor: fiber contained}.
Additionally we suppose that $V$ is indecomposable (although this
assumption will be removed in the next subsection). Then $\phi_{n}^{-1}(\phi_{n}\gamma_{n})$
is universally homeomorphic to the quotient of $\AA_{k}^{e+\sht_{V}(j)}$
by some linear $G$-action. \end{prop}
\begin{proof}
Let $M_{n}:=\cO_{E^{*}}/\fm_{E}^{np+1}$ and let $M_{n}^{\natural}\subset M_{n}$
be the image of $\cO_{E^{*}}^{\delta^{d}=0}=\bigoplus_{i=0}^{d-1}k((t))g^{i}.$
Then $M_{n,e}^{\natural}\subset M_{n}^{\natural}$ and we can identify
$\Hom[\cO_{E}][\natural]{\beta^{*}\Omega_{S/R}}{M_{n,e}}$ with similarly
defined $\Hom[\cO_{E}][\natural]{\beta^{*}\Omega_{S/R}}{M_{n}}.$
Indeed since $\beta^{*}\Omega_{S/R}$ has length $pe,$ every $\cO_{E}$-linear
map $\beta^{*}\Omega_{S/R}\to M_{n}$ has its image in $M_{n,e}.$
Therefore we may count the dimension of $\Hom[\cO_{E}][\natural]{\beta^{*}\Omega_{S/R}}{M_{n}}$
instead. This trick will make following arguments easier. 

Let $\omega_{f_{1}},\dots,\omega_{f_{d}}$ be as in Lemma \ref{lem:nice basis omega's}.
Then 
\[
\Hom[\cO_{E}][\natural]{\beta^{*}\Omega_{S/R}}{M_{n}}=\left\{ \alpha\in\Hom[\cO_{E}][\natural]{\beta\Omega_{S/k}}{M_{n}}\mid\alpha(\omega_{f_{i}})=0\,(i=1,\dots,d)\right\} .
\]
Identifying $\Hom[\cO_{E}][\natural]{\beta^{*}\Omega_{S/k}}{M_{n}}$
with $M_{n}^{\natural}$, we can identify $\Hom[\cO_{E}][\natural]{\beta^{*}\Omega_{S/R}}{M_{n}}$
with the set of $h\in M_{n}^{\natural}$ satisfying:
\begin{equation}
\begin{pmatrix}\delta_{-}^{d-1}(f_{1}) & \delta_{-}^{d-2}(f_{1}) & \dots & \delta_{-}(f_{1}) & f_{1}\\
 & \delta_{-}^{d-2}(f_{2}) & \dots & \delta_{-}(f_{2}) & f_{2}\\
 &  & \ddots & \vdots & \vdots\\
 &  &  & \delta_{-}(f_{d-1}) & f_{d-1}\\
 &  &  &  & f_{d}
\end{pmatrix}\begin{pmatrix}h\\
\delta(h)\\
\vdots\\
\delta^{d-2}(h)\\
\delta^{d-1}(h)
\end{pmatrix}=0\mod\fm_{E}^{np+1}.\label{eq:matrix}
\end{equation}
Let us write $h=\sum_{\lambda=0}^{d-1}h^{[\lambda]}g^{\lambda}$,
$h^{[\lambda]}\in k((t))$ such that if we write $h^{[\lambda]}=\sum_{i}h_{i}^{[\lambda]}t^{i}$,
then $h_{i}^{[\lambda]}=0$ for $i$ with 
\[
pi-j\lambda>np\,\left(\Leftrightarrow i>n+\left\lfloor \frac{j\lambda}{p}\right\rfloor \right).
\]
Then 
\[
\delta^{d-i}(h)=\sum_{\lambda=d-i}^{d-1}\delta^{d-i}(g^{\lambda})\cdot h^{[\lambda]}.
\]
We note that $v_{E}(\delta^{d-i}(g^{d-i}))=0$. From the bottom row
of equation (\ref{eq:matrix}), 
\[
f_{d}\cdot\delta^{d-1}(g^{d-1})\cdot h^{[d-1]}=0\mod\fm_{E}^{np+1}.
\]
Hence the coefficients $h_{i}^{[d-1]}$ of $h_{i}^{[d-1]}$ are zero
for $i$ such that 
\[
v_{E}(f_{d})+ip\le np\,\left(\Leftrightarrow i\le n-\frac{v_{E}(f_{d})}{p}\right).
\]
The other $v_{E}(f_{d})/p+\left\lfloor (d-1)j/p\right\rfloor $ coefficients,
\[
h_{i}^{[d-1]},\, n-\frac{v_{E}(f_{d})}{p}<i\le n+\left\lfloor \frac{(d-1)j}{p}\right\rfloor ,
\]
can take arbitrary values. Suppose now that $h^{[d-1]},\dots,h^{[d-l]}$
are fixed. Then we consider the $(l+1)$-th row from the bottom of
equation (\ref{eq:matrix}),
\[
\delta_{-}^{l}(f_{d-l})\cdot\delta^{d-l-1}(g^{d-l-1})\cdot h^{[d-l-1]}+(\text{fixed terms})=0\mod\fm_{E}^{np+1}.
\]
The left hand side is the image of $\omega_{f_{d-l}}$ by the $G$-equivariant
map $\alpha\in\Hom[\cO_{E}]{\beta^{*}\Omega_{S/k}}{M_{n}}$ with $dx_{1}\mapsto h$.
In particular, it is $G$-invariant for any $h.$ Hence the equality
holds for at least one choice of $h^{[d-l-1]}\in k((t)).$ Indeed
we can choose $h^{[d-l-1]}$ as 
\[
-\frac{\text{(fixed terms)}}{\delta_{-}^{l}(f_{d-l})\cdot\delta^{d-l-1}(g^{d-l-1})}\in k((t))
\]
with coefficients in degree $>n+\left\lfloor j(d-l-1)/p\right\rfloor $
eliminated. (This is the point where we use the trick of replacing
$M_{n,e}$ with $M_{n}$.) Once a solution exists, then the equation
uniquely determines the coefficients $h_{i}^{[d-l-1]}$ of $h^{[d-l-1]}$
for $i$ such that
\[
v_{E}(\delta_{-}^{l}(f_{d-l}))+ip\le np\,\left(\Leftrightarrow i\le n-\frac{v_{E}(\delta_{-}^{l}(f_{d-l}))}{p}\right).
\]
The other $v_{E}(\delta_{-}^{l}(f_{d-l}))/p+\left\lfloor (d-l-1)j/p\right\rfloor $
coefficients, 
\[
h_{i}^{[d-l-1]},\, n-\frac{v_{E}(\delta_{-}^{l}(f_{d-l}))}{p}<i\le n+\left\lfloor \frac{(d-l-1)j}{p}\right\rfloor 
\]
can take arbitrary values. Hence the solution space of (\ref{eq:matrix})
has dimension
\[
\sum_{l=0}^{d-1}\left(\frac{v_{E}(\delta_{-}^{l}(f_{d-l}))}{p}+\left\lfloor \frac{(d-l-1)j}{p}\right\rfloor \right)=e+\sht_{V}(j).
\]
We have completed the proof. 
\end{proof}

\subsubsection{Counting the dimension of $\protect\Hom[\cO_{E}][\natural]{\beta^{*}\Omega_{S/R}}{M_{n,e}}$:
The decomposable case}

We generalize the shift number to the decomposable case as follows:
\begin{defn}
Suppose $V=\bigoplus_{\lambda=1}^{l}V_{d_{\lambda}}.$ Then we define
the \emph{shift number }of $j\in\NN'_{0}$ \emph{with respect to $V$}
to be 
\[
\sht_{V}(j):=\sum_{\lambda=1}^{l}\sht_{V_{d_{\lambda}}}(j)=\sum_{\lambda=1}^{l}\sum_{i=1}^{d_{\lambda}-1}\left\lfloor \frac{ij}{p}\right\rfloor .
\]

\end{defn}
With this definition, Proposition \ref{prop:key dim count} holds
also in the decomposable case. Now we sketch how the above arguments
can be generalized to this case. Let 
\[
k[\bx]=k[x_{\lambda,i}\mid1\le\lambda\le l,\,1\le i\le d_{\lambda}]
\]
be the coordinate ring of $V$ as in Section \ref{sub:Our-setting}.
Then $\Im((\psi\circ\beta)^{*}\Omega_{R/k}\to\beta^{*}\Omega_{S/k})$
is again generated by $G$-invariant elements. A $G$-invariant element
of $\beta^{*}\Omega_{S/k}$ is of the form
\[
\omega_{\bf}:=\sum_{\lambda}\sum_{i}\delta_{-}^{d_{\lambda}-i}(f_{\lambda})dx_{i}
\]
for some 
\[
\bf=(f_{1},\dots,f_{l})\in\prod_{\lambda=1}^{l}\cO_{E}^{\delta_{-}^{d_{\lambda}}=0}.
\]
Then we can generalize Lemma \ref{lem:nice basis omega's} as follows:
\begin{lem}
There exist a basis of $ $$\Im((\psi\circ\beta)^{*}\Omega_{R/k}\to\beta^{*}\Omega_{S/k})$,
\[
\omega_{\bf_{\lambda,i}}\,(1\le\lambda\le l,\,1\le i\le d_{\lambda}),
\]
such that the terms of $dx_{\lambda',i'}$ in $\omega_{\bf_{\lambda,i}}$
vanish if either ``$\lambda'<\lambda$'' or ``$\lambda'=\lambda$
and $i'<i$'' . \end{lem}
\begin{proof}
The proof is almost the same as the one of Lemma \ref{lem:nice basis omega's}.
We use induction on $l.$ We first take any basis $\omega_{\bf_{\lambda,i}}.$
Next we eliminate the $dx_{1,1}$ terms of $\omega_{\bf_{\lambda,i}}$
with $(\lambda,i)\ne(1,1),$ suitably replacing $\bf_{\lambda,i}$'s.
Then we eliminate the $dx_{1,2}$ terms of $\omega_{\bf_{\lambda,i}}$
with $(\lambda,i)\ne(1,1),\,(1,2).$ Repeating this procedure $d_{1}$
times, we get a basis such that $\omega_{\bf_{1,i}}$, $1\le i\le d_{1},$
satisfy the condition of the assertion. Applying the assumption of
induction to $\omega_{\bf_{\lambda,i}},$ $\lambda\ge2,$ we prove
the lemma. 
\end{proof}
Once the above lemma is proved, the rest of arguments in the indecomposable
case work also in the decomposable case. Hence:
\begin{prop}
\label{prop:key dim count decomposable}The assertion of Proposition
\ref{prop:key dim count} holds without the assumption that $V$ is
indecomposable.
\end{prop}

\subsection{The change of variables formula}
\begin{defn}
We define a function $\fs_{\cX}$ on $\cJ_{\infty}\cX$ as the composition
$\sht_{V}\circ\rj$: 
\[
\fs_{\cX}:\cJ_{\infty}\cX\xrightarrow{\rj}\NN'_{0}\xrightarrow{\sht_{V}}\ZZ.
\]
\end{defn}
\begin{lem}
\label{lem:pre (change vars)}Let $A\subset J_{\infty}X$ be a cylinder.
Then the function, $-\ord\Jac_{\phi}-\fs_{\cX}$, on $\phi_{\infty}^{-1}(A)$
is exponentially measurable and
\[
\int_{\phi_{\infty}^{-1}(A)}\LL^{-\ord\Jac_{\phi}-\fs_{\cX}}d\mu_{\cX}=\mu_{X}(A).
\]
\end{lem}
\begin{proof}
Let $\cY\subset\cX$ be the exceptional locus of $\phi:\cX\to X.$
Then there exists a stratification $\phi_{\infty}^{-1}(A)\setminus\cJ_{\infty}\cY=\bigsqcup_{i\in\NN}B_{i}$
such that
\begin{enumerate}
\item for every $i,$ $B_{i}$ is a cylinder, and
\item for every $i,$ the functions, $\ord\Jac_{\phi}$, $\ord\phi^{-1}\Jac_{X}$
and $\rj$, are constant on $B_{i}.$
\end{enumerate}
From Lemma \ref{cor:image of cylinder}, for every $i>0,$ $\phi_{\infty}(B_{i})$
is a cylinder. Let $n:=-\ord\Jac_{\phi}(B_{i})-\fs_{\cX}(B_{i})$.
Then from Propositions \ref{prop:key dim count} and \ref{prop:key dim count decomposable},
$\mu_{\cX}(B_{i})\LL^{n}=\mu_{X}(\phi_{\infty}(B_{i})).$ Hence 
\[
\mu_{X}(A)=\sum_{i}\mu_{X}(\phi_{\infty}(B_{i}))=\sum_{i}\int_{B_{i}}\LL^{-\ord\Jac_{\phi}-\fs_{\cX}}d\mu_{\cX}.
\]
 Now the lemma follows from Proposition \ref{prop:integrable_pieces}.\end{proof}
\begin{thm}[The change of variables formula, cf. \cite{Denef-Loeser_McKay,Yasuda Twisted_jets,Yasuda Motivic_Over_DM}]
\label{thm:change vars}Let $A\subset J_{\infty}X$ be a strongly
measurable subset and let $B=\bigsqcup_{i}B_{i}\subset\cJ_{\infty}\cX$
be a countable disjoint union of strongly measurable subsets $B_{i}$
such that $B=\phi_{\infty}^{-1}(A).$ (For instance, if $A$ is a
cylinder, then $B=\phi_{\infty}^{-1}(A)$ satisfies this condition.)
Let $F:J_{\infty}X\supset A\to\frac{1}{r}\ZZ\cup\{\infty\}$ be a
measurable function. Then $F$ is exponentially integrable if and
only if the function, $F\circ\phi_{\infty}-\ord\Jac_{\phi}-\fs_{\cX}$,
on $B$ is exponentially integrable. Moreover if it is the case, then
we have
\[
\int_{A}\LL^{F}d\mu_{X}=\int_{B}\LL^{F\circ\phi_{\infty}-\ord\Jac_{\phi}-\fs_{\cX}}d\mu_{\cX}.
\]
\end{thm}
\begin{proof}
Using Proposition \ref{prop:integrable_pieces}, we will deduce the
situation to an easier one step by step. 
\begin{enumerate}
\item We may suppose that $B$ is disjoint from $\cJ_{\infty}\cY$: From
Lemma \ref{lem:measure zero}, $\cJ_{\infty}\cY$ is measurable and
of measure zero. Also $J_{\infty}Y$ is measurable and of measure
zero. Moreover $B_{i}\setminus\cJ_{\infty}\cY$ and $A\setminus J_{\infty}Y$
are strongly measurable. Therefore we may replace $B$ with $B\setminus\cJ_{\infty}\cY$,
and $A$ with $A\setminus J_{\infty}Y.$ 
\item We may suppose that $B$ is strongly measurable: From Corollary \ref{cor:image of str measurable},
for each $i$, $\phi_{\infty}(B_{i})$ is strongly measurable. From
Proposition \ref{prop:integrable_pieces}, it is enough to prove the
theorem for $B=B_{i}$ and $A=\phi_{\infty}(B_{i}).$
\item We may suppose that $B$ is a cylinder and $A$ is stable: Let $B_{0}(\epsilon)\subset B$
be cylinders as in Definition \ref{def: measurable subset}. Then
for a sequence $\epsilon_{i}\in\RR_{>0},$ $i\in\NN$ with $\lim_{i\to\infty}\epsilon_{i}=0,$
$B\setminus\bigcup_{i\in\NN}B_{0}(\epsilon_{i})$ is measurable and
has measure zero. Its image in $A$ is also measurable and has measure
zero. Hence it is enough to prove the theorem for $B=B_{0}(\epsilon_{i})$
and $A=\phi_{\infty}(B_{0}(\epsilon)).$ As we have supposed that
$B$ is disjoint from $\cJ_{\infty}\cY,$ $B_{0}(\epsilon_{i})$ is
also disjoint from $\cJ_{\infty}\cY$. Hence $\phi_{\infty}(B_{0}(\epsilon))$
is stable.
\item We may suppose that the functions, $\ord\Jac_{\phi},$ $\ord\phi^{-1}\Jac_{X}$
and $\rj,$ are constant on $B$: We take the stratification $B=\bigsqcup_{i}B_{i}$
of $B$ with respect to the values of these functions. Then the stratification
has only finite strata. It suffices to show the theorem for $B=B_{i}$
and $A=\phi_{\infty}(B_{i}).$
\end{enumerate}
Now we fix $n\in\frac{1}{r}\ZZ\cup\{\infty\}.$ Then $C:=F^{-1}(n)$
is measurable. Let $C_{i}(\epsilon)$ be as in Definition \ref{def: measurable subset}.
For every $i$ and $\epsilon,$ we can take $C_{i}(\epsilon)\subset A.$
Let $e:=\ord\Jac_{\phi}(B)$ and $s:=\fs_{\cX}(B).$ Then 
\[
\mu_{\cX}(\phi_{\infty}^{-1}(C_{i}(\epsilon)))=\mu_{X}(C_{i}(\epsilon))\LL^{-e-s}.
\]
This shows that 
\[
\lim_{\epsilon\to0}\mu_{\cX}(\phi_{\infty}^{-1}(C_{0}(\epsilon))=\mu_{\cX}(\phi_{\infty}^{-1}(C)).
\]
Hence
\begin{eqnarray*}
\int_{C}\LL^{F}d\mu_{X} & = & \mu_{X}(C)\LL^{n}\\
 & = & \lim_{\epsilon\to0}\mu_{X}(C_{0}(\epsilon))\LL^{n}\\
 & = & \lim_{\epsilon\to0}\mu_{\cX}(\phi_{\infty}^{-1}(C_{0}(\epsilon)))\LL^{n-e-s}\\
 & = & \mu_{\cX}(\phi_{\infty}^{-1}(C))\LL^{n-e-s}\\
 & = & \int_{\phi_{\infty}^{-1}(C)}\LL^{F\circ\phi_{\infty}-\ord\Jac_{\phi}-\fs_{\cX}}d\mu_{\cX}.
\end{eqnarray*}
Now the theorem follows again from Proposition \ref{prop:integrable_pieces}.
\end{proof}

\section{Stringy invariants and the McKay correspondence\label{sec:Stringy-invariants}}

In this section, we will define stringy invariants as certain motivic
integrals. Then we will obtain several versions of the McKay correspondence,
which are consequences of the change of variables formula above.

\subsection{Stringy invariants of stringily Kawamata log terminal pairs}

We will keep the notation. Let $\omega_{X}$ be the canonical sheaf
of $X$, which is defined as the double dual of $\bigwedge^{d}\Omega_{X/k}$.
Since $R$ is a UFD (see for instance \cite[Theorem 3.8.1]{Campbell-Wehlau}),
in particular, $X$ is 1-Gorenstein. Namely $\omega_{X}$ is invertible.
The $\omega$\emph{-Jacobian ideal }$\Jac_{X}^{\omega}\subset\cO_{X}$
of $X$ is defined by the equality 
\[
\mathrm{Im}\left(\bigwedge^{d}\Omega_{X/k}\to\omega_{X}\right)=\Jac_{X}^{\omega}\cdot\omega_{X}.
\]

\begin{rem}
It is known that if $X$ is a local complete intersection, then $\Jac_{X}=\Jac_{X}^{\omega}.$
In our setting, this is the case only when $V$ is isomorphic to $V_{2}\oplus V_{1}^{\oplus d-2}$,
$V_{2}^{\oplus2}\oplus V_{1}^{\oplus d-4}$ or $V_{3}\oplus V_{1}^{\oplus d-3}.$
Indeed, $X$ is a hypersurface in these cases (see for instance \cite{Campbell-Wehlau}).
Otherwise $X$ is not even Cohen-Macaulay from \cite{Ellingsrud Skjelbred}.
\end{rem}
In \cite{Batyrev_Gor,Batyrev-NonArch}, Batyrev introduced \emph{stringy
invariants }for Kawamata log terminal pairs in characteristic zero.
To define it, he used the resolution of singularities. However we
are working in positive characteristic and it is not known yet that
there always exists a resolution of singularities. Therefore, following
Denef and Loeser \cite{Denef-Loeser_McKay}, we will define the stringy
invariant as an integral on $J_{\infty}X.$
\begin{defn}
Let $Z=\sum_{i=1}^{m}a_{i}Z_{i}$ be a formal $\QQ$-linear combination
of closed subschemes $Z_{i}\subsetneq X$. Then we define a function
\[
\ord Z:=\sum_{i=1}^{m}a_{i}\cdot\ord Z_{i}:J_{\infty}X\to\QQ\cup\{\infty\}.
\]
Here we suppose that $\ord Z$ takes the constant value $\infty$
on the measure zero subset $\bigcup_{i=1}^{m}(\ord Z_{i})^{-1}(\infty)$.
We say that the pair $(X,Z)$ is \emph{stringily Kawamata log terminal
}if the function $\ord Z+\ord\Jac_{X}^{\omega}$ on $J_{\infty}X$
is exponentially integrable. If it is the case, we define its \emph{stringy
motivic invariant }by\emph{
\[
M_{\st}(X,Z):=\int_{J_{\infty}X}\LL^{\ord Z+\ord\Jac_{X}^{\omega}}d\mu_{X}.
\]
}Then we define the \emph{stringy Poincaré function} by 
\[
P_{\st}(X,Z)=P(M_{\st}(X,Z))\in\bigcup_{r=1}^{\infty}\ZZ((T^{-1/r})).
\]

\end{defn}
Let $f:Y\to X$ be a proper birational morphism with $Y$ smooth.
Then we define the \emph{canonical divisor} of $f$, denoted $K_{f}$,
to be the divisor such that 
\[
\Im(f^{*}\omega_{X}\to\omega_{Y}\otimes K(Y))=\cO_{Y}(-K_{f})\cdot\omega_{Y}.
\]
Here $K(Y)$ denotes the function field of $Y.$ This divisor has
support in the exceptional locus of $f.$ We define the \emph{pull-back}
$f^{*}Z$ of $Z$ by $f$ as the formal linear combination $\sum_{i=1}^{m}a_{i}\cdot f^{-1}Z_{i}$
of the scheme-theoretic preimages $f^{-1}Z_{i}\subset Y.$
\begin{prop}
With the notation as above, we have
\[
M_{\st}(X,Z)=M_{\st}(Y,f^{*}Z-K_{f}).
\]
\end{prop}
\begin{proof}
Since
\[
f^{-}\Jac_{X}^{\omega}\cdot\cO_{Y}(-K_{f})=\Jac_{f},
\]
we have
\[
f^{-1}\ord\Jac_{X}^{\omega}-\ord\Jac_{f}=-\ord K_{f}.
\]
The proposition follows from this and the change of variables formula
for varieties in positive characteristic \cite{Sebag_motivic_formal}. \end{proof}
\begin{cor}
\label{cor:Explicit formula}Suppose that $K_{f}-f^{*}Z$ is a simple
normal crossing divisor written as $\sum_{i=1}^{m}a_{i}E_{i},$ $a_{i}\in\QQ,$
with $E_{i}$ prime divisors. Then $(X,Z)$ is stringily kawamata
log terminal if and only if for every $i,$ $a_{i}>-1.$ Moreover
if it is the case, then
\[
M_{\st}(X,Z)=M_{\st}(Y,f^{*}Z-K_{f})=\sum_{I\subset\{1,\dots,m\}}[E_{I}^{\circ}]\prod_{i\in I}\frac{\LL-1}{\LL^{1+a_{i}}-1},\,\text{where }E_{I}^{\circ}:=\bigcap_{i\in I}E_{i}\setminus\bigcup_{i\notin I}E_{i}.
\]
\end{cor}
\begin{proof}
The corollary follows from the preceding proposition and the explicit
computation of motivic integrals (for instance, see \cite{Craw}). \end{proof}
\begin{defn}
\label{def: discrep KLT}Let $(X,Z)$ be a pair as above. Let $f:Y\to X$
be a proper birational morphism such that $Y$ is normal and $f^{*}Z$
is a Cartier divisor. Let $E$ be a prime divisor on $Y$. Then the
canonical divisor $K_{f}$ of $f$ is similarly defined on the smooth
locus of $Y$ (and extends to $Y$ as a Weil divisor). The \emph{discrepancy
}of $(X,Z)$ at $E$, denoted $a(E;X,Z),$ is defined to be the coefficient
of $E$ in the divisor $K_{f}-f^{*}Z.$ We say that $(X,Z)$ is \emph{Kawamata
log terminal }if for every $Y$ and $E$ as above, $a(E;X,Z)>-1.$\end{defn}
\begin{prop}
If $(X,Z)$ is stringily Kawamata log terminal, then it is Kawamata
log terminal. Additionally, if there exists a resolution $f:Y\to X$
with $K_{f}-f^{*}Z$ a simple normal crossing $\QQ$-$ $Cartier divisor,
then the converse is also true. \end{prop}
\begin{proof}
For the first assertion, let $f:Y\to X$ and $E$ be as in Definition
\ref{def: discrep KLT}. Let $U\subset Y$ be an open dense subset
such that 
\[
(K_{f}-f^{*}Z)|_{U}=a(E;X,Z)\cdot E|_{U}.
\]
Then from the change of variables formula in \cite{Sebag_motivic_formal},
$\ord f^{*}Z-\ord K_{f}$ is exponentially integrable on $J_{\infty}U.$
Hence $a(E;X,Z)>-1$. This shows that $(X,Z)$ is Kawamata log terminal.

The second assertion follows from Corollary \ref{cor:Explicit formula}.
\end{proof}
Next we will define stringy invariants of stacky pairs. 
\begin{defn}
Let $\cZ=\sum_{i=1}^{m}a_{i}\cZ_{i}$ be a formal $\QQ$-linear combination
of closed substacks $Z_{i}\subsetneq\cX$. Then we define a function
\[
\ord Z:=\sum_{i=1}^{m}a_{i}\cdot\ord Z_{i}:\cJ_{\infty}\cX\to\QQ\cup\{\infty\}.
\]
We say that the pair $(\cX,\cZ)$ is \emph{stringily Kawamata log
terminal }if the function $\ord\cZ-\fs_{\cX}$ on $\cJ_{\infty}\cX$
is exponentially integrable. If it is the case, we define its \emph{stringy
motivic invariant by
\[
M_{\st}(\cX,\cZ):=\int_{\cJ_{\infty}\cX}\LL^{\ord\cZ-\fs_{\cX}}d\mu_{\cX}.
\]
}
\end{defn}

\subsection{An explicit formula for $M_{\st}(\cX)$}
\begin{defn}
For a $G$-representation $V=\bigoplus_{\lambda=1}^{l}V_{d_{\lambda}},$
we put 
\begin{gather*}
D_{V}:=\sum_{\lambda=1}^{l}\frac{(d_{\lambda}-1)d_{\lambda}}{2}=\sum_{\lambda=1}^{l}\sum_{i=1}^{d_{\lambda}-1}i\in\NN.
\end{gather*}
\end{defn}
\begin{prop}
\label{prop:explict M_st} The pair $(\cX,0)$ is stringily Kawamata
log terminal if and only if $D_{V}\ge p.$ Moreover if it is the case,
\[
M_{\st}(\cX):=M_{\st}(\cX,0)=\LL^{d}+\frac{\LL^{l-1}(\LL-1)\left(\sum_{s=1}^{p-1}\LL^{s-\sht_{V}(s)}\right)}{1-\LL^{p-1-D_{V}}}.
\]
\end{prop}
\begin{proof}
For an integer $s$ with $1\le s\le p-1$ and a non-negative integer
$n,$ we have 
\begin{eqnarray*}
\sht_{V}(np+s) & = & \sum_{\lambda=1}^{l}\sum_{i=1}^{d_{\lambda}-1}\left(in+\left\lfloor \frac{is}{p}\right\rfloor \right)\\
 & = & \left(\sum_{\lambda=1}^{l}\frac{(d_{\lambda}-1)d_{\lambda}}{2}\right)n+\sum_{\lambda=1}^{l}\sum_{i=1}^{d_{\lambda}-1}\left\lfloor \frac{is}{p}\right\rfloor \\
 & = & D_{V}\cdot n+\sht_{V}(s).
\end{eqnarray*}
Hence we have 
\begin{eqnarray*}
M_{\st}(\cX) & = & \int_{\cJ_{\infty}\cX}\LL^{-\fs_{\cX}}d\mu_{\cX}\\
 & = & \sum_{j\in\NN'_{0}}\mu_{\cX}(\cJ_{\infty,j}\cX)\LL^{-\sht_{V}(j)}\\
 & = & \LL^{d}+\sum_{j\in\NN'}[\AA_{k}^{l}\times\bGCovrep{D,j}]\LL^{-\sht_{V}(j)}\\
 & = & \LL^{d}+\sum_{j>0}(\LL-1)\LL^{l+j-1-\left\lfloor j/p\right\rfloor -\sht_{V}(j)}\\
 & = & \LL^{d}+(\LL-1)\LL^{l-1}\sum_{s=1}^{p-1}\sum_{n=0}^{\infty}\LL^{(p-1-D_{V})n+s-\sht_{V}(s)}.
\end{eqnarray*}
This converges if and only if $D_{V}\ge p.$ Now the formula of the
proposition easily follows.\end{proof}
\begin{cor}
If $D_{V}=p,$ then 
\[
M_{\st}(\cX)=\LL^{d}+\LL^{l}\sum_{s=1}^{p-1}\LL^{s-\sht_{V}(s)}\in\ZZ[\LL].
\]
\end{cor}
\begin{proof}
The equality is a direct consequence of the preceding Proposition.
Moreover in this case, 
\[
\sht_{V}(s)\le\frac{s}{p}\sum_{\lambda=1}^{l}\sum_{i=1}^{d_{\lambda}-1}\lambda=\frac{s}{p}D_{V}=s.
\]
Hence $M_{\st}(\cX)\in\ZZ[\LL].$\end{proof}
\begin{defn}
When $D_{V}\ge p,$ we define the \emph{stringy Euler number }of $\cX$
by
\[
e_{\st}(\cX):=e_{\top}(M_{\st}(\cX)).
\]
(See Section \ref{sub:Localization-and-completion}.)\end{defn}
\begin{cor}
\label{cor:explicit e_st}If $D_{V}\ge p,$ then 
\[
e_{\st}(\cX)=1+\frac{p-1}{D_{V}-p+1}.
\]
In particular, if $D_{V}=p,$ then $e_{\st}(\cX)=p.$\end{cor}
\begin{proof}
Obvious.
\end{proof}
In the following, we compute $M_{\st}(\cX)$ in several cases. 
\begin{example}
\label{ex:lVp}Suppose that $V=V_{p}^{\oplus l}$. Then $\sht_{V}(s)=l(s-1)(p-1)/2$
(see for instance \cite[page 94]{Graham-Knuth-Patashnik}). Hence
\[
M_{\st}(\cX)=\LL^{d}+(\LL-1)\LL^{l-1}\sum_{s=1}^{p-1}\frac{\LL^{s-l(s-1)(p-1)/2}}{1-\LL^{p-1-lp(p-1)/2}}.
\]

\end{example}
In the following cases, we have $D_{V}=p.$
\begin{example}
\label{ex: V_3}If $p=3$ and $V=V_{3},$ then $M_{\st}(\cX)=\LL^{3}+2\LL^{2}.$
\begin{example}
\label{ex:2V_2}If $p=2$ and $V=V_{2}^{\oplus2},$ then $M_{\st}(\cX)=\LL^{4}+\LL^{3}.$
\begin{example}
\label{ex: pV2 Mst}If $V=V_{2}^{\oplus p},$ then 
\[
M_{\st}(\cX)=\LL^{2p}+\LL^{p}(\LL^{p-1}+\LL^{p-2}+\cdots+\LL)=\LL^{2p}+\LL^{2p+1}+\cdots+\LL^{p+1}.
\]

\end{example}
\end{example}
\end{example}

\subsection{The McKay correspondence}

We say that $V$ has \emph{reflections }if the fixed point locus $V^{G}$
has codimension one. A non-trivial $G$-representation $V$ has reflection
if and only if $V\cong V_{2}\oplus V_{1}^{\oplus d-2}.$

\subsubsection{The no reflection case}
\begin{lem}
\label{lem:Jac=00003DJac}Suppose that $V$ has no reflection. Then
$\phi^{-1}\Jac_{X}^{\omega}=\Jac_{\phi}.$\end{lem}
\begin{proof}
Since the quotient map $\psi:V\to X$ is étale in codimension one,
$\psi^{*}\omega_{X}=\omega_{V},$ and the lemma follows. \end{proof}
\begin{cor}
\label{cor:McKay-pair}Let $Z=\sum_{i}a_{i}Z_{i}$ be a formal $\QQ$-linear
combination of closed subschemes $Z_{i}\subsetneq X.$ Then $(X,Z)$
is stringily Kawamata log terminal if and only if so is $(\cX,\phi^{*}Z)$.
Moreover if it is the case, then 
\[
M_{\st}(X,Z)=M_{\st}(\cX,\phi^{*}Z).
\]
\end{cor}
\begin{proof}
From the change of variables formula and the preceding lemma, we have
\[
M_{\st}(X,Z)=\int_{J_{\infty}X}\LL^{\ord Z+\ord\Jac_{X}^{\omega}}d\mu_{X}=\int_{\cJ_{\infty}\cX}\LL^{\ord\phi^{*}Z-\fs_{\cX}}d\mu_{\cX}=M_{\st}(\cX,\phi^{*}Z).
\]
\end{proof}
\begin{cor}
\label{cor:McKay-non-pair}The pair $(X,0)$ is stringily Kawamata
log terminal if and only if $D_{V}\ge p.$ Moreover if it is the case,
then 
\[
M_{\st}(X)=M_{\st}(\cX).
\]
\end{cor}
\begin{proof}
This follows from the preceding corollary and Proposition \ref{prop:explict M_st}.\end{proof}
\begin{rem}
In particular, if $D_{V}\ge p,$ then $X$ has only canonical singularities.
Namely all discrepancies are non-negative. On the other hand, from
\cite{Yasuda "pure subrings"}, $X$ does not satisfy a closely related
property, the strong F-regularity.\end{rem}
\begin{cor}[The $p$-cyclic McKay correspondence]
\label{cor:McKay-crepant}Suppose that $V$ has no reflection. and
that there exists a crepant resolution $f:Y\to X$. Then the following
hold:
\begin{enumerate}
\item $M_{\st}(\cX)=[Y]$.
\item $D_{V}=p.$
\item $e_{\top}(Y)=p.$
\end{enumerate}
\end{cor}
\begin{proof}
For the first assertion, we have 
\[
M_{\st}(\cX)=M_{\st}(X)=[Y].
\]
Then, in particular, $\cX$ is stringily Kawamata log terminal. Hence
$d\ge p.$ Also from the first assertion, we have 
\[
e_{\st}(\cX)=e_{\top}(Y),
\]
which is an integer. Now the second and third assertions follows from
Corollary \ref{cor:explicit e_st}. \end{proof}
\begin{rem}
\label{rem: non-linear}The third assertion of Corollary \ref{cor:McKay-crepant}
does not hold if we replace $V$ with a \emph{non-linear} $G$-action
on $\Spec k[[x_{1},\dots,x_{d}]].$ For instance, let $0\in X$ be
either a $E_{8}^{2}$-singularity in characteristic two, a $E_{6}^{1}$-singularity
in characteristic three, or a $E_{8}^{1}$-singularity in characteristic
five (for the notation, see \cite{Artin RDP}). Then $X$ is the quotient
of $\Spec k[[x,y]]$ by a $G$-action such that the associated covering
$\Spec k[[x,y]]\to X$ is étale outside 0. The minimal resolution
of $X$ is a crepant resolution. (Indeed from \cite[Th. 4.1]{Lipman}
and \cite[Cor. 4.19]{Badescu}, the blowup of $X$ at the singular
point has only rational double points. Hence discrepancies at the
exceptional curves are zero.) However the topological Euler characteristic
of the minimal resolution is not $p.$\end{rem}
\begin{example}
\label{ex: V3 crepant}Suppose that $V=V_{3}.$ If we suppose that
$G$ acts on the coordinate ring $k[x,y,z]$ of $V,$ by $x\mapsto x,$
$y\mapsto-x+y,$ $z\mapsto x-y+z,$ then the invariant subring is
\[
k[x,y,z]^{G}=k[x,N_{y},N_{z},d],
\]
where $d=y^{2}+xz-xy,$ $N_{y}=\prod_{g\in G}g(y),$ $N_{z}=\prod_{g\in G}g(z)$
(see \cite[Th. 4.10.1]{Campbell-Wehlau}). In particular, $X$ is
a hypersurface. By abuse of symbols, we let $x,N_{y},N_{z},d$ correspond
to variables $X,Y,Z,W$ respectively. Then according to computations
with Macaulay2 \cite{Macaulay2} for small primes $p$, the defining
equation of $X$ seems to be 
\[
2X^{p}Z+W^{p}-Y^{2}+\sum_{i=2}^{(p+1)/2}(-1)^{i}C_{i-1}X^{2(p-i)}W^{i}.
\]
Here $C_{i}$ denotes the $i$-th Catalan number modulo $p$. The
author does not know if this equation is known. 

Now suppose $p=3.$ Then the equation becomes 
\[
-X^{3}Z+W^{3}-Y^{2}+X^{2}W^{2}=0.
\]
If $k$ is algebraically closed, then by a suitable coordinate transformation,
$X$ is defined by 
\[
Z^{2}+X^{3}+Y^{4}+X^{2}Y^{2}+Y^{3}W=0.
\]
This equation defines the compound $E_{6}^{1}$-singularity as studied
by Hirokado, Ito and Saito \cite{HIS 3d canonical}. The blowup $X_{1}$
of $X$ along the singular locus is singular along a line. Then the
blowup $Y$ of $X_{1}$ along the singular locus is now smooth and
a crepant resolution of $X$. Moreover the exceptional locus of $Y\to X$
is a simple normal crossing divisor with two irreducible components,
say $E_{1}\cup E_{2}$. Then $E_{i}\cong\AA_{k}^{1}\times\PP_{k}^{1}$
and $E_{1}\cap E_{2}\cong\AA_{k}^{1}.$ Therefore $[Y]=\LL^{3}+2\LL^{2},$
which agrees with our previous computation in Example \ref{ex: V_3}. 
\begin{example}
\label{ex: V2++V2}Suppose that $p=2$ and $V=V_{2}\oplus V_{2}.$
Using the description of the invariant subring in \cite[Th. 1.12.1]{Campbell-Wehlau},
we can see that $X$ is a hypersurface defined by
\[
W^{2}X+V^{2}Y+VWZ+Z^{2}.
\]
By direct computation, the blowup $Y$ of $X$ along the singular
locus is a crepant resolution and its exceptional locus is isomorphic
to $\AA_{k}^{2}\times\PP_{k}^{1}$. Therefore 
\[
M_{\st}(X)=[Y]=\LL^{4}+\LL^{3},
\]
which coincides with Example \ref{ex:2V_2}.
\end{example}
\end{example}
\begin{cor}
\textup{\label{cor:McKay over the origin}Suppose that $V$ has no
reflection. Let $f:Y\to X$ be a crepant resolution and $E_{0}:=f^{-1}(0).$
Then, with the notation of Section \ref{sub:Motivic-integration-on-GCov},
we have 
\[
\int_{\bGCovrep D}\LL^{-\sht_{V}}d\nu=[E_{0}].
\]
Moreover if $k$ is a finite field, then for each finite extension
$\FF_{q}/k$, we have
\[
\sharp E_{0}(\FF_{q})=\sum_{j\in\NN'_{0}}\frac{\sharp\bGCovrep{D,j}(\FF_{q})}{q{}^{\sht_{V}(j)}}.
\]
}\end{cor}
\begin{proof}
We have:
\begin{eqnarray*}
\int_{\bGCovrep D}\LL^{-\sht_{V}}d\nu & = & \int_{(\cJ_{\infty}\cX)_{0}}\LL^{-\fs_{\cX}}d\mu_{\cX}\\
 & = & \int_{\pi_{0}^{-1}(0)}\LL^{\ord\Jac_{X}^{\omega}}d\mu_{X}\\
 & = & \int_{\pi_{0}^{-1}(E_{0})}1\, d\mu_{Y}\\
 & = & [E_{0}]
\end{eqnarray*}
This shows the first assertion. Then the second assertion is obtained
by applying the counting-points realization $\sharp_{q}.$\end{proof}
\begin{example}
Let $k=\FF_{2}$ and $V=V_{2}^{\oplus2}$. Let $Y\to X$ be a crepant
resolution as in Example \ref{ex: V2++V2}. Then $E_{0}=\PP_{k}^{1}$
and for each finite extension $\FF_{q}/\FF_{2},$ $\sharp E_{0}(\FF_{q})=q+1.$
On the other hand, 
\begin{eqnarray*}
\sum_{j\in\NN'_{0}}\frac{\sharp\bGCovrep{D,j}(\FF_{q})}{q^{\sht_{V}(j)}} & = & 1+\sum_{n=0}^{\infty}\frac{\bGCovrep{D,2n+1}(\FF_{q})}{q^{\sht_{V}(2n+1)}}\\
 & = & 1+\sum_{n=0}^{\infty}(q-1)\cdot q^{-n}\\
 & = & 1+q.
\end{eqnarray*}

\end{example}
The second assertion of Corollary \ref{cor:McKay over the origin}
will be slightly differently formulated in terms of Artin-Schreier
extensions of $k((t)).$ We have
\[
\sharp\bGCovrep{D,j}(\FF_{q})=\frac{1}{p}\sharp\GCov{D\times_{k}\FF_{q},j}.
\]
Let $N_{q,j}$ be the number of Galois extensions of $\FF_{q}((t))$
in a fixed algebraic closure $\overline{k((t))}$ of $k((t))$ with
ramification jump $j.$ Then for $j\in\NN',$ we have 
\[
\sharp\GCov{D\times_{k}\FF_{q},j}=(p-1)N_{q,j}.
\]
The factor, $p-1,$ comes from $p-1$ choices of isomorphism of $G$
to the Galois group. Hence:
\begin{cor}
\textup{\label{cor:counting AS exts}With the same notation as above,
we have
\[
\sharp E_{0}(\FF_{q})=1+\frac{p-1}{p}\sum_{j\in\NN'}\frac{N_{q,j}}{q{}^{\sht_{V}(j)}}.
\]
}
\end{cor}

\subsubsection{The reflection case}

Next we consider the case where $V$ has reflections. We write the
coordinate ring of $V$ as $S=k[x,y,z_{1},\dots,z_{d-2}]$ with the
$G$-action given by:
\begin{gather*}
\sigma(x)=x+y\\
\sigma(y)=y\\
\sigma(z_{i})=z_{i}
\end{gather*}
Then 
\[
R:=S^{G}=k[x^{p}-xy^{p-1},y,z_{1},\dots,z_{d-2}].
\]
(See for instance \cite[Th. 1.11.2]{Campbell-Wehlau}.) Thus $X$
is again an affine $d$-space, in particular, smooth. 
\begin{lem}
\label{lem:reflection Jac}We have $\Jac_{\psi}=(y^{p-1})\subset S.$ \end{lem}
\begin{proof}
From the explicit generators above of $R,$ the module $\Omega_{S/R}$
is isomorphic to the cokernel of the map $S^{d}\to S^{d}$ represented
by the Jacobian matrix
\[
\begin{pmatrix}-y^{p-1}\\
(1-p)xy^{p-2} & 1\\
 &  & 1\\
 &  &  & \ddots\\
 &  &  &  & 1
\end{pmatrix}.
\]
Hence by definition, $\Jac_{\psi}$ is generated by its determinant
$-y^{p-1},$ which proves the lemma.
\end{proof}
The fixed point locus $V^{G}$ is defined by the ideal $(y)\subset S.$
We define $\cY$ to be the quotient stack $[V^{G}/G]$, which is a
closed reduced substack of $\cX$ and define $Y$ to be the image
of $\cY$ on $X,$ which is defined by $(y)\subset R.$ 
\begin{cor}
\label{cor:McKay-pair-1}Suppose that $V$ has reflections. Let $Z=\sum a_{i}Z_{i}$
be a formal $\QQ$-linear combination of closed subschemes $Z_{i}\subsetneq X.$
Then $(X,Z)$ is stringily Kawamata log terminal if and only if so
is $(\cX,\phi^{*}Z+(1-p)\cY).$ Moreover, if it is the case, then
\[
M_{\st}(X,Z)=M_{\st}(\cX,\phi^{*}Z+(1-p)\cY).
\]
In particular, for $a<1,$ 
\[
M_{\st}(X,aY)=M_{\st}(\cX,(a+1-p)\cY).
\]
\end{cor}
\begin{proof}
In this case, since $X$ is smooth, $\Jac_{X}^{\omega}=R.$ Hence
\begin{eqnarray*}
M_{\st}(X,Z) & = & \int_{J_{\infty}X}\LL^{\ord Z}d\mu_{X}\\
 & = & \int_{\cJ_{\infty}\cX}\LL^{\ord\phi^{*}Z-\ord\Jac_{\phi}-\fs_{\cX}}d\mu_{\cX}\\
 & = & \int_{\cJ_{\infty}\cX}\LL^{\ord\phi^{*}Z-\ord(p-1)\cY-\fs_{\cX}}d\mu_{\cX}\\
 & = & M_{\st}(\cX,\phi^{*}Z+(1-p)\cY).
\end{eqnarray*}
\end{proof}
\begin{example}
We suppose $V=V_{2}.$ From Corollary \ref{cor:Explicit formula},
for $a<1,$ we have 
\[
M_{\st}(X,aY)=\frac{\LL^{2}-\LL}{1-\LL^{a-1}}.
\]
On the other hand, for $a<2-p,$ we can compute $M_{\st}(\cX,a\cY)$
from the definition as follows: First we have:
\begin{eqnarray*}
M_{\st}(\cX,a\cY) & = & \int_{\cJ_{\infty,0}\cX}\LL^{\ord a\cY}d\mu_{\cX}+\sum_{j>0}\int_{\cJ_{\infty,j}}\LL^{\ord a\cY-\fs_{\cX}}d\mu_{\cX}\\
 & = & \frac{\LL^{2}-\LL}{1-\LL^{a-1}}+\sum_{j\in\NN'}\int_{\cJ_{\infty,j}}\LL^{\ord a\cY-\fs_{\cX}}d\mu_{\cX}.
\end{eqnarray*}
Now we compute the second term. Let us fix $E\in\bGCovrep{D,j}(\bar{k})$
and for $G$-arcs $\gamma:E\to V,$ we denote associated twisted arcs
by $\bar{\gamma}.$ Then, with the notation as in Section \ref{subsec: detail act},
those $G$-arcs $\gamma:E\to V$ with $\ord\cY(\bar{\gamma})=n$ correspond,
by $\gamma\mapsto\gamma^{*}(x),$ to
\[
\bar{k}[[t]]\times\bar{k}^{*}\cdot gt^{n}\times\prod_{m>n}\bar{k}\cdot gt^{m},
\]
provided $np-j>0.$ Otherwise there is no such $G$-arc. Hence for
$np>j,$ 

\begin{eqnarray*}
\mu_{\cX}((\ord\cY)^{-1}(n)\cap\cJ_{\infty,j}\cX) & = & (\LL-1)\LL^{n+1+\left\lfloor j/p\right\rfloor }\times(\LL-1)\LL^{j-1-\left\lfloor j/p\right\rfloor }\\
 & = & \LL^{-n+j}(\LL-1)^{2}.
\end{eqnarray*}
Hence
\begin{align*}
\sum_{j\in\NN'}\int_{\mathcal{J}_{\infty,j}\mathcal{X}}\mathbb{L}^{\ord a\cY-\fs_{\cX}}d\mu & =\sum_{n=1}^{\infty}\sum_{j\in\NN'}\mu_{\cX}((\ord\cY)^{-1}(n)\cap\mathcal{J}_{\infty,j}\mathcal{X})\mathbb{L}^{an-\left\lfloor j/p\right\rfloor }\\
 & =\sum_{n=1}^{\infty}\sum_{\substack{j\in\NN'\\
j<np
}
}\mathbb{L}^{(a-1)n+j-\left\lfloor j/p\right\rfloor }(\mathbb{L}-1)^{2}.
\end{align*}
Then, putting $j=rp+i$ ($0\le r<n,$ $0<i<p)$, we continue:
\begin{eqnarray*}
 & = & (\mathbb{L}-1)^{2}\sum_{n=1}^{\infty}\left(\sum_{i=1}^{p-1}\sum_{r=0}^{n-1}\mathbb{L}^{(a-1)n+rp+i-r}\right)\\
 & = & (\mathbb{L}-1)^{2}\sum_{n=1}^{\infty}\sum_{i=1}^{p-1}\mathbb{L}^{(a-1)n+i}\frac{1-\mathbb{L}^{n(p-1)}}{1-\mathbb{L}^{p-1}}\\
 & = & (\mathbb{L}-1)\mathbb{L}\sum_{n=1}^{\infty}\mathbb{L}^{(a-1)n}(\mathbb{L}^{n(p-1)}-1)\\
 & = & (\mathbb{L}-1)\mathbb{L}\left(\frac{\mathbb{L}^{a+p-2}}{1-\mathbb{L}^{a+p-2}}-\frac{\mathbb{L}^{a-1}}{1-\mathbb{L}^{a-1}}\right).
\end{eqnarray*}
As expected, we now have
\begin{eqnarray*}
M_{\st}(\cX,a\cY) & = & \frac{\LL^{2}-\LL}{1-\LL^{a-1}}+(\mathbb{L}-1)\mathbb{L}\left(\frac{\mathbb{L}^{a+p-2}}{1-\mathbb{L}^{a+p-2}}-\frac{\mathbb{L}^{a-1}}{1-\mathbb{L}^{a-1}}\right)\\
 & = & \frac{\LL^{2}-\LL}{1-\LL^{a+p-2}}.
\end{eqnarray*}
\end{example}
\begin{rem}
For the pair $(\cX,a\cY)$ being stringily Kawamata log terminal,
the coefficient $a$ must be negative. In particular, $(\cX,0)$ is
not stringily Kawamata log terminal, and its stringy invariant $M_{\st}(\cX)$
is not defined in this paper. However, stringy invariants should be
generalized beyond log terminal singularities to some extent, as Veys
\cite{Veys} confirmed for surface singularities in characteristic
zero. Then, it appears meaningful to, for instance, claim
\[
M_{\st}(\cX,a\cY)=\frac{\LL^{2}-\LL}{1-\LL^{a+p-2}}\text{ and }e_{\st}(\cX,a\cY)=\frac{2}{2-a-p},
\]
unless $a+p-2=0.$ In this way, we would be able to relate weighted
counts of Artin-Schreier extensions of $k((t))$ to stringy invariants
of singularities even when $D_{V}<p.$
\end{rem}

\subsection{Pseudo-projectivization and the Poincaré duality}

Batyrev proved that the stringy invariant of a log terminal projective
variety in characteristic zero satisfies the \emph{Poincaré duality.
}We will obtain a similar result for the ``projectivization'' of
our quotient variety $X$. 

Let $\GG_{m}=\Spec k[t^{\pm}]$ be the multiplicative group scheme
over $k$. We have natural $\GG_{m}$-actions on $V$ and $X$, which
are compatible with the quotient map $V\to X.$ Moreover the action
on $V$ commutes with the $G$-action, inducing a $\GG_{m}\times G$-action
on $V.$ Let $ $$\cW_{1}$ and $\cW_{2}$ be the quotient stacks
$[(V\setminus\{0\})/\GG_{m}\times G]$ and $[(X\setminus\{0\})/\GG_{m}]$
respectively. The former is a smooth Deligne-Mumford stack and isomorphic
to $[\PP(V)/G]$. The latter is not a variety but a (singular) Artin
stack with finite stabilizers. Indeed, from the following lemma, the
$\GG_{m}$-action on $X\setminus\{0\}$ have non-reduced stabilizers
$\Spec k[t]/(t^{p}-1)\subset\GG_{m}$ at singular points. 
\begin{lem}
Let $W\subset V$ be the fixed point locus of the $G$-action and
$\bar{W}\subset X$ its image. Then the morphism $W\to\bar{W}$ is
isomorphic to the Frobenius morphism of $W.$ \end{lem}
\begin{proof}
We prove only the indecomposable case. Let $k[x_{1},\dots,x_{d}]$
be the coordinate ring of $V$ with the $G$-action as in Section
\ref{sub:Our-setting}. Then $W$ is defined by $x_{2}=\cdots=x_{d}=0.$
Hence the coordinate ring $k[\bar{W}]$ of $\bar{W}$ is identified
with the image of $k[x_{1},\dots,x_{d}]^{G}$ on $k[x_{1}]=k[x_{1},\dots,x_{d}]/(x_{2},\dots,x_{d}).$
We easily see that $x_{1}\not\in k[\bar{W}].$ On the other hand,
$x_{1}^{p},$ the image of the norm of $x_{1}$, is in $k[\bar{W}].$
Since $k[W]$ is purely inseparable over $k[\bar{W}],$ we have $k[\bar{W}]=k[x_{1}^{p}],$
which shows the lemma. 
\end{proof}
The stacks $\cW_{1}$ and $\cW_{2}$ have the same coarse moduli space
$\bar{X}=(V\setminus\{0\})/\GG_{m}\times G.$ Moreover $\cW_{1}$
and $\cW_{2}$ have open dense subsets isomorphic to $(V\setminus V^{G})/\GG_{m}$
and are birational. Since the morphisms $\cX\setminus\{0\}\to\cW_{1}$
and $X\setminus\{0\}\to\cW_{2}$ are $\GG_{m}$-torsors, it seems
natural to define stringy invariants of $\cW_{1}$ and $\cW_{2}$
as follows.
\begin{defn}
We define the \emph{stringy motivic invariant }of $\cW_{1}$ and $\cW_{2}$
by:
\[
M_{\st}(\cW_{1})=M_{\st}(\cW_{2}):=\frac{\LL^{d}-\LL^{l}}{\LL-1}+(M_{\st}(X)-(\LL^{d}-\LL^{l}))\frac{\LL^{l}-1}{\LL^{l}(\LL-1)}.
\]

\end{defn}
Let $\cW$ denote either $\cW_{1}$ or $\cW_{2}.$ 
\begin{prop}
\label{prop:projective stringy}Suppose that $D_{V}\ge p.$ Then 
\[
M_{\st}(\cW)=\frac{\LL^{d}-1}{\LL-1}+\frac{(\LL^{l}-1)\left(\sum_{s=1}^{p-1}\LL^{s-\sht_{V}(s)}\right)}{\LL(1-\LL^{p-1-D_{V}})}.
\]
\end{prop}
\begin{proof}
We have
\begin{eqnarray*}
M_{\st}(\cW) & = & \frac{\LL^{d}-\LL^{l}}{\LL-1}+\left(\LL^{d}+\frac{\LL^{l-1}(\LL-1)\left(\sum_{s=1}^{p-1}\LL^{s-\sht_{V}(s)}\right)}{1-\LL^{p-1-D_{V}}}-(\LL^{d}-\LL^{l})\right)\frac{\LL^{l}-1}{\LL^{l}(\LL-1)}\\
 & = & \frac{\LL^{d}-1}{\LL-1}+\frac{(\LL^{l}-1)\left(\sum_{s=1}^{p-1}\LL^{s-\sht_{V}(s)}\right)}{\LL(1-\LL^{p-1-D_{V}})}.
\end{eqnarray*}
\end{proof}
\begin{prop}[Poincaré duality]
Let us write $M_{\st}(\cW)$ as $M_{\st}(\cW;\LL)$ to clarify that
it is a rational function in $\LL.$ Then we have
\[
M_{\st}(\cW;\LL^{-1})\LL^{d-1}=M_{\st}(\cW;\LL).
\]
\end{prop}
\begin{proof}
The first term of the expression in Proposition \ref{prop:projective stringy}
equals $[\PP_{k}^{d-1}],$ which obviously satisfies the Poincaré
duality. For the second term, substituting $\LL^{-1}$ for $\LL$
and multiplying with $\LL^{d-1}$, we obtain
\[
\frac{(\LL^{-l}-1)\left(\sum_{s=1}^{p-1}\LL^{\sht_{V}(s)-s}\right)}{\LL^{-1}(1-\LL^{-p+1+D_{V}})}\LL^{d-1}=\frac{(\LL^{l}-1)\left(\sum_{s=1}^{p-1}\LL^{\sht_{V}(s)-s}\right)\LL^{p+d-l-D_{V}}}{\LL(1-\LL^{p-1-D_{V}})}.
\]
Then the Poincaré duality follows from the following equations: For
$1\le s\le p-1,$ 
\begin{align*}
 & \sht_{V}(p-s)-(p-s)+p+d-l-D_{V}\\
 & =s+\left(\sum_{\lambda=1}^{l}\sum_{i=1}^{d_{\lambda}-1}i+\left\lfloor -\frac{is}{p}\right\rfloor \right)+d-l-D_{V}\\
 & =s+\left(\sum_{\lambda=1}^{l}\sum_{i=1}^{d_{\lambda}-1}-\left\lfloor \frac{is}{p}\right\rfloor -1\right)+d-l\\
 & =s-\sht_{V}(s).
\end{align*}

\end{proof}

\section{Remarks on future problems\label{sec:Comments} }

\subsection{Generalizations}

This study should be a toy model for the \emph{wild McKay correspondence.}
The following are possible directions of generalization. 
\begin{enumerate}
\item General groups and non-linear actions: If we similarly define twisted
arcs, then the almost bijection between twisted arcs of $\cX$ and
arcs of $X$ should be valid in very general. Looking at Harbater's
work \cite{Harbater}, we should be able to construct the spaces of
twisted arcs or jets at least for $p$-groups, whether their detailed
structure can be understood or not. As explained in Remark \ref{rem: non-linear},
the non-linear case will be quite different from the linear case even
in dimension two. Some non-linear action appears as the projectivization
of a linear one. Then we may apply some results in the linear case
to such cases. 
\item General local fields: Sebag \cite{Sebag_motivic_formal} generalized
the motivic integration to formal schemes over a discrete valuation
ring. Replacing $k((t))$ with a general local field along this line,
we might be able to get, for instance, a result on weighted counts
of Galois extensions of the local field. 
\item General proper birational morphisms of general Deligne-Mumford stacks:
We proved the change of variables formula only for the morphism $[V/G]\to V/G.$
However, ultimately, it should be generalized to an arbitrary proper
birational morphism of Deligne-Mumford stacks (with a mild finiteness
condition). It was obtained in \cite{Yasuda Motivic_Over_DM} when
the morphism is tame and stacks are smooth. 
\end{enumerate}

\subsection{Other related problems}

As we saw in Corollary \ref{cor:McKay-crepant}, if there exists a
crepant resolution of $X,$ then $D_{V}=p.$ What about the converse?
The only known examples of crepant resolutions of $X$ are Examples
\ref{ex: V_3} and \ref{ex:2V_2}. For instance, if $V=V_{2}^{\oplus p},$
then from Example \ref{ex: pV2 Mst}, $M_{\st}(X)=\LL^{2p}+\LL^{p}\cdot[\PP^{p-1}]$.
This seems to suggest that there exists a crepant resolution $Y\to X$
such that the exceptional locus is isomorphic to $\AA_{k}^{p}\times\PP_{k}^{p-1}.$

In characteristic zero, the McKay correspondence is proved at the
level of derived category \cite{BKR,Craw-Ishii,Bezrukavnikov-Kaledin,Kawamata}.
However, from \cite{Yi}, the skew group algebra $k[\bx]*G$ always
has \emph{infinite }global dimension in the wild case. Then, if there
exists something like the \emph{derived wild McKay correspondence},
then what would replace the derived category of $k[\bx]*G$-modules?


\begin{thebibliography}{99}
\bibitem{Artin RDP}M.~Artin. \newblock Coverings of the rational double points in characteristic {$p$}. \newblock In {\em Complex analysis and algebraic geometry}, pages 11--22.   Iwanami Shoten, Tokyo, 1977.


\bibitem{Badescu}L.~B{\u{a}}descu. \newblock {\em Algebraic surfaces}. \newblock Universitext. Springer-Verlag, New York, 2001. \newblock Translated from the 1981 Romanian original by Vladimir Ma{\c{s}}ek   and revised by the author.


\bibitem{Batyrev_Gor}V.~V. Batyrev. \newblock Stringy {H}odge numbers of varieties with {G}orenstein canonical   singularities. \newblock In {\em Integrable systems and algebraic geometry (Kobe/Kyoto,   1997)}, pages 1--32. World Sci. Publ., River Edge, NJ, 1998.


\bibitem{Batyrev-NonArch}V.~V. Batyrev. \newblock Non-{A}rchimedean integrals and stringy {E}uler numbers of   log-terminal pairs. \newblock {\em J. Eur. Math. Soc. (JEMS)}, 1(1):5--33, 1999.


\bibitem{Bezrukavnikov-Kaledin}R.~V. Bezrukavnikov and D.~B. Kaledin. \newblock Mc{K}ay equivalence for symplectic resolutions of quotient   singularities. \newblock {\em Tr. Mat. Inst. Steklova}, 246(Algebr. Geom. Metody, Svyazi i   Prilozh.):20--42, 2004.


\bibitem{BKR}T.~Bridgeland, A.~King, and M.~Reid. \newblock The {M}c{K}ay correspondence as an equivalence of derived categories. \newblock {\em J. Amer. Math. Soc.}, 14(3):535--554 (electronic), 2001.


\bibitem{Campbell-Wehlau}H.~E. A.~E. Campbell and D.~L. Wehlau. \newblock {\em Modular invariant theory}, volume 139 of {\em Encyclopaedia of   Mathematical Sciences}. \newblock Springer-Verlag, Berlin, 2011. \newblock Invariant Theory and Algebraic Transformation Groups, 8.


\bibitem{Craw}A.~Craw. \newblock An introduction to motivic integration. \newblock In {\em Strings and geometry}, volume~3 of {\em Clay Math. Proc.},   pages 203--225. Amer. Math. Soc., Providence, RI, 2004.


\bibitem{Craw-Ishii}A.~Craw and A.~Ishii. \newblock Flops of {$G$}-{H}ilb and equivalences of derived categories by   variation of {GIT} quotient. \newblock {\em Duke Math. J.}, 124(2):259--307, 2004.


\bibitem{Denef-Loeser_Germs}J.~Denef and F.~Loeser. \newblock Germs of arcs on singular algebraic varieties and motivic   integration. \newblock {\em Invent. Math.}, 135(1):201--232, 1999.


\bibitem{Denef-Loeser_McKay}J.~Denef and F.~Loeser. \newblock Motivic integration, quotient singularities and the {M}c{K}ay   correspondence. \newblock {\em Compositio Math.}, 131(3):267--290, 2002.


\bibitem{Ellingsrud Skjelbred}G.~Ellingsrud and T.~Skjelbred. \newblock Profondeur d'anneaux d'invariants en caract{\'e}ristique {$p$}. \newblock {\em Compositio Math.}, 41(2):233--244, 1980.


\bibitem{GSp and Verdier}G.~Gonzalez-Sprinberg and J.-L. Verdier. \newblock Sur la r{\`e}gle de {M}c{K}ay en caract{\'e}ristique positive. \newblock {\em C. R. Acad. Sci. Paris S{\'e}r. I Math.}, 301(11):585--587,   1985.


\bibitem{Graham-Knuth-Patashnik}R.~L. Graham, D.~E. Knuth, and O.~Patashnik. \newblock {\em Concrete mathematics}. \newblock Addison-Wesley Publishing Company Advanced Book Program, Reading, MA,   1989. \newblock A foundation for computer science.


\bibitem{Macaulay2}D.~R. Grayson and M.~E. Stillman. \newblock Macaulay 2, a software system for research in algebraic geometry. \newblock Available at http://www.math.uiuc.edu/Macaulay2/.


\bibitem{EGA-I} A.~Grothendieck and J.~A. Dieudonn{\'e}. \newblock {\em {{\'E}l{\'e}ments de g{\'e}om{\'e}trie alg{\'e}brique. I.}} \newblock {Die Grundlehren der mathematischen Wissenschaften. 166.   Berlin-Heidelberg-New York: Springer-Verlag. IX, 466 p. }, 1971.


\bibitem{Harbater}D.~Harbater. \newblock Moduli of {$p$}-covers of curves. \newblock {\em Comm. Algebra}, 8(12):1095--1122, 1980.


\bibitem{HIS 3d canonical}M. Hirokado, H. Ito and N. Saito, Three
dimensional canonical singularities in codimension two in positive
characteristic, in preparation. preprint.

\bibitem{Kawamata}Y.~Kawamata. \newblock Log crepant birational maps and derived categories. \newblock {\em J. Math. Sci. Univ. Tokyo}, 12(2):211--231, 2005.


\bibitem{Lipman}J.~Lipman. \newblock Rational singularities, with applications to algebraic surfaces and   unique factorization. \newblock {\em Inst. Hautes {\'E}tudes Sci. Publ. Math.}, (36):195--279, 1969.


\bibitem{Looijenga}E.~Looijenga. \newblock Motivic measures. \newblock {\em Ast{\'e}risque}, (276):267--297, 2002. \newblock S{{\'e}}minaire Bourbaki, Vol. 1999/2000.


\bibitem{Milne book}J.~S. Milne. \newblock {\em \'{E}tale cohomology}, volume~33 of {\em Princeton Mathematical   Series}. \newblock Princeton University Press, Princeton, N.J., 1980.


\bibitem{Nicaise-trace}J.~Nicaise. \newblock A trace formula for varieties over a discretely valued field. \newblock {\em J. Reine Angew. Math.}, 650:193--238, 2011.


\bibitem{Nicase-Sebag}J.~Nicaise and J.~Sebag. \newblock A note on motivic integration in mixed characteristic. \newblock arXiv:0912.4887.


\bibitem{Reid bourbaki}M.~Reid. \newblock La correspondance de {M}c{K}ay. \newblock {\em Ast{\'e}risque}, (276):53--72, 2002. \newblock S{{\'e}}minaire Bourbaki, Vol. 1999/2000.


\bibitem{Rose}M.~A. Rose. \newblock Frobenius action on {$l$}-adic {C}hen-{R}uan cohomology. \newblock {\em Commun. Number Theory Phys.}, 1(3):513--537, 2007.


\bibitem{Schroer}S.~Schr{{\"o}}er. \newblock The {H}ilbert scheme of points for supersingular abelian surfaces. \newblock {\em Ark. Mat.}, 47(1):143--181, 2009.


\bibitem{Sebag_motivic_formal}J.~Sebag. \newblock Int{\'e}gration motivique sur les sch{\'e}mas formels. \newblock {\em Bull. Soc. Math. France}, 132(1):1--54, 2004.


\bibitem{Serre-local}J.-P. Serre. \newblock {\em Local fields}, volume~67 of {\em Graduate Texts in Mathematics}. \newblock Springer-Verlag, New York, 1979. \newblock Translated from the French by Marvin Jay Greenberg.


\bibitem{Thomas Lara}L.~Thomas. \newblock A valuation criterion for normal basis generators in equal positive   characteristic. \newblock {\em J. Algebra}, 320(10):3811--3820, 2008.


\bibitem{Veys}W.~Veys. \newblock Stringy invariants of normal surfaces. \newblock {\em J. Algebraic Geom.}, 13(1):115--141, 2004.


\bibitem{Yasuda Twisted_jets}T.~Yasuda. \newblock Twisted jets, motivic measures and orbifold cohomology. \newblock {\em Compos. Math.}, 140(2):396--422, 2004.


\bibitem{Yasuda Motivic_Over_DM}T.~Yasuda. \newblock Motivic integration over {D}eligne-{M}umford stacks. \newblock {\em Adv. Math.}, 207(2):707--761, 2006.


\bibitem{Yasuda "pure subrings"}T. Yasuda, Pure subrings of regular
local rings, endomorphism rings and Frobenius morphisms, arXiv:1108.5797,
to appear in J. Algebra. 

\bibitem{Yi}Z.~Yi. \newblock Homological dimension of skew group rings and crossed products. \newblock {\em J. Algebra}, 164(1):101--123, 1994.
\end{thebibliography}
\end{document}